\documentclass{amsart}
\usepackage[utf8]{inputenc}
\usepackage[margin=1in]{geometry}

\usepackage{lmodern}
\usepackage{mathtools}
\usepackage{xcolor}
\usepackage{amsthm}
\usepackage{amsmath}
\usepackage{comment}
\usepackage{ stmaryrd }
\usepackage{url}

\usepackage{amsaddr}
\usepackage{doi}

\usepackage{hyperref}
\usepackage{appendix}

\usepackage{amsfonts}
\usepackage{amssymb}

\newcommand{\R}{\mathbb{R}}
\newcommand{\N}{\mathbb{N}}
\newcommand{\Sb}{\mathbb{S}}

\newcommand{\E}{\mathcal{E}}
\newcommand{\M}{\mathcal{M}}
\newcommand{\HH}{\mathcal{H}}
\newcommand{\D}{\mathcal{D}}
\newcommand{\Deps}{\mathcal{D}_{\varepsilon}}

\newcommand{\F}{\mathcal{F}}

\newcommand{\Me}{{\M_{\varepsilon}^f}}
\newcommand{\A}{\mathcal{A}}

\newcommand{\dd}{\mathrm{d}}

\newcommand{\e}{\varepsilon}
\newcommand{\vi}{\varphi}

\newtheorem{theorem}{Theorem}
\newtheorem{proposition}{Proposition}
\newtheorem{definition}{Definition}

\newtheorem{lemma}[proposition]{Lemma}
\newtheorem{remark}{Remark}
\newtheorem{corollary}[proposition]{Corollary}

\title[Extending Cercignani's conjecture to Nordheim]{Extending Cercignani's conjecture results from Boltzmann to Boltzmann-Fermi-Dirac equation}
\author{Thomas Borsoni}
\address{Sorbonne Université, CNRS, Université Paris Cité, Laboratoire Jacques-Louis Lions (LJLL), F-75005 Paris, France}
\email{thomas.borsoni@sorbonne-universite.fr}

\thanks{\noindent Data sharing not applicable to this article as no datasets were generated or analysed during the current study.
}
\date{\today}

\keywords{kinetic theory, Nordheim equation, Boltzmann equation, Boltzmann-Fermi-Dirac, Fermi-Dirac statistics, entropy, entropy methods, Cercignagni's conjecture, Csiszár-Kullback-Pinsker inequality}

\begin{document}

\maketitle

\vspace{-.5cm}

\begin{abstract}
We establish a connection between the relative Classical entropy and the relative Fermi-Dirac entropy, allowing to transpose, in the context of the Boltzmann or Landau equation, any entropy-entropy production inequality from one case to the other; therefore providing entropy-entropy production inequalities for the Boltzmann-Fermi-Dirac operator, similar to the ones of the Classical Boltzmann operator. We also provide a generalized version of the Csiszár-Kullback-Pinsker inequality to weighted $L^p$ norms, $1 \leq p \leq 2$ and a wide class of entropies.
\end{abstract}

\vspace{-.5cm}

\tableofcontents

\section{Introduction}

\medskip

\noindent In the study of quantum dilute gases, Nordheim~\cite{nordhiem1928kinetic}, Uehling and Uhlenbeck~\cite{uehling1933transport} proposed in the early 1930s a kinetic approach, with a modified Boltzmann equation - bearing the names Nordheim equation, quantum Boltzmann equation, or Boltzmann-Fermi-Dirac equation (for fermions). Quoting Villani~\cite[Chapter 5.3]{villanireview}, this description "does not rest on the traditional quantum formalism, [...] it is rather a classical description of interacting particles with quantum features", which nevertheless allows to recover the expected Fermi-Dirac statistics (in the context of fermions) and Bose-Einstein statistics (in the context of bosons).
In this paper, we focus on the case of fermions in which Pauli's exclusion principle is taken into account.

\medskip

\noindent In the context of the study of trend to equilibrium, the use of entropy-entropy production inequalities, allowing to obtain explicit rates of convergence, has been increasing in the recent years. We mention the results of Villani~\cite{villani2003} for the Classical Boltzmann operator, Desvillettes~\cite{desvillettescoulomb}, Desvillettes and Villani~\cite{desvillettesvillani2}, Carrapatoso, Desvillettes and He~\cite{carrapatoso} for the Landau operator, Alonso, Bagland, Desvillettes and Lods~\cite{ABDLentropy,ABDLsoft} and Alonso, Bagland, and Lods~\cite{ABL} for the Landau-Fermi-Dirac (LFD) operator. We also refer the reader to the reviews~\cite{villanireview,devillettes2011}. Up to now, no such results were known for the Boltzmann-Fermi-Dirac (BFD) operator, and the initial goal of this work was to fill this gap. In doing so, we actually established a connection between the Classical and  Fermi-Dirac cases, allowing to transpose any entropy-entropy production result from one setting to the other under suitable conditions. Entropy-entropy production estimates for the BFD operator then directly follow from the ones of Villani~\cite{villani2003} for the Classical Boltzmann operator - the same procedure can be applied to transpose entropy-entropy production estimates for the Landau operator to ones for the LFD operator.

\medskip

    \noindent This paper is constructed as follows. In the rest of this section, we introduce our mathematical setting, useful definitions and further discuss our motivations. In Section~\ref{section:mainresults}, we state our main results, which are the connection between the Fermi-Dirac and Classical relative entropies (Theorem~\ref{theorem:FD} and Proposition~\ref{prop:FDupper}), the entropy-entropy production estimates for the BFD operator (Corollary~\ref{theorem:BFDcerci}), and a generalization of the Csiszár-Kullback-Pinsker inequality (Proposition~\ref{theorem:CKnonopti}). We provide the proof of Theorem~\ref{theorem:FD} in Section~\ref{section:prooftheoremFD}, of Corollary~\ref{theorem:BFDcerci} in Section~\ref{section:proofcorollarybfd}, discuss the case of the Landau / LFD equation in Section~\ref{section:LFD} and provide the proof of Proposition~\ref{theorem:CKnonopti} (along with a more precise result and discussions) in Section~\ref{appendix:CK}. We very briefly discuss the case of bosons in Appendix~\ref{appendix:boseeinstein}. Considerations about a wider class of entropies (with a generalization of Proposition~\ref{theorem:CKnonopti}) are provided in Appendix~\ref{appendix:entropystudy}. In particular, we provide in Proposition~\ref{prop:general} a general formula for the minimizer of a general entropy under the constraints given by general conserved quantities, under a sole condition of existence. There are links with the results of~\cite{bredendesvillettes} where equilibria of collision kernels of type “two particles give two particles” appearing in weak turbulence theory are rigorously obtained. In fact, Proposition~\ref{prop:general} does not replace a rigorous study of equilibria, such as in~\cite{bredendesvillettes}, but gives insights, in a general setting, on to which distributions we should be expected to be equilibria. Finally, technical results relevant to our study are proven in Appendix~\ref{appendix:technicallemmas}.

\medskip

\subsection{The Boltzmann-Fermi-Dirac operator} The spatially homogeneous Boltzmann-Fermi-Dirac (BFD) equation writes
\begin{equation} \label{eq:BFDequation}
    \partial_t f = Q_{\varepsilon}(f,f), \qquad f(0,\cdot) = f^{in},
\end{equation}
where $f \equiv f(t,v) \geq 0$ represents a density of fermions (quantum particles of half-integer spin, e.g. electrons), depending on time $t \geq 0$ and velocity $v \in \R^3$, $f^{in} \equiv f^{in}(v) \geq 0$ is an initial distribution and $Q_{\varepsilon}$ stands for the BFD operator. We refer the reader to~\cite{CCbook,dolbeaultFD,lu2001spatially,luwennberg} for results on the Cauchy problem associated to~\eqref{eq:BFDequation}. In this paper, we focus on the operator $Q_{\varepsilon}$ and, if not specified otherwise, discard the time variable $t$, considering $f \equiv f(v) \geq 0$.

\medskip

\noindent In dimension $3$, the BFD operator writes, for $\varepsilon > 0$ and measurable $f : \R^3 \to \left[0, \varepsilon^{-1} \right]$,
\begin{equation} \label{eq:BFDoperator}
    Q_{\varepsilon}(f,f)(v) := \iint_{\R^3 \times \Sb^{2}} \Big[f' f'_* (1 - \varepsilon f) (1 - \varepsilon f_*) - f f_* (1 - \varepsilon f') (1 - \varepsilon f'_*) \Big] B(v,v_*,\sigma) \, \dd \sigma \, \dd v_*,
\end{equation}
where we used the standard shorthands $f \equiv f(v)$, $f_* \equiv f(v_*)$, $f' \equiv f(v')$, $f'_* \equiv f(v'_*)$ and
\begin{equation}
    v' = \frac{v+v_*}{2} + \frac{|v-v_*|}{2} \, \sigma, \qquad v'_* = \frac{v+v_*}{2} - \frac{|v-v_*|}{2} \, \sigma.
\end{equation}
We insist that any physically meaningful solution $f$ to \eqref{eq:BFDequation} should be such that $1 - \varepsilon f~\geq~0$. As a consequence of time micro-reversibility and Galilean invariance properties, the collision kernel $B \geq 0$ is assumed to satisfy, for $(v,v_*) \in \R^3 \times \R^3$ and  $\sigma \in \Sb^{2}$,
\begin{equation} \label{eqassump:B}
B(v,v_*,\sigma) \equiv B(|v-v_*|,\cos \theta),
\end{equation}
with $\displaystyle \cos \theta = \sigma \cdot \frac{v-v_*}{|v-v_*|}$. At this level of generality, no other assumption is made on $B$. Formally choosing $\varepsilon = 0$ in \eqref{eq:BFDoperator} yields the Classical Boltzmann operator
\begin{equation} \label{eq:Boltzoperator}
    Q_{0}(f,f)(v) := \iint_{\R^3 \times \Sb^{2}} \Big[f' f'_* - f f_*\Big] B(v,v_*,\sigma) \, \dd \sigma \, \dd v_*,
\end{equation}
with the same shorthands as in \eqref{eq:BFDoperator}. Now for $\varepsilon \geq 0$, define
\begin{equation} \label{eqdef:varphieps}
\varphi_{\varepsilon}(x) := \frac{x}{1-\varepsilon x}, \qquad \text{for } x \in \R_+ \; \; \text{ s.t. } \; \; 1 - \varepsilon x > 0.
\end{equation}
Remark that when $\varepsilon = 0$, corresponding to the classical case, we have $\varphi_0 = \mathrm{Id}_{\R_+}$. Moreover, $\varphi_{\varepsilon}$ is a $\mathcal{C}^1$-diffeomorphism from $[0,\varepsilon^{-1})$ to $\R_+$, with ${\varphi_{\varepsilon}}^{-1} = \varphi_{-\varepsilon}$.

\medskip

\noindent Then, when $f$ is such that $1 - \e f > 0$ almost everywhere, Equation \eqref{eq:BFDoperator} can be rewritten in the following form resembling the Classical Boltzmann operator \eqref{eq:Boltzoperator},
\begin{equation} \label{eq:BFDrewritten}
 Q_{\varepsilon}(f,f)(v) := \iint_{\R^3 \times \Sb^{2}} \Big[\varphi_{\varepsilon}(f)' \, \varphi_{\varepsilon}(f)'_* - \varphi_{\varepsilon}(f) \varphi_{\varepsilon}(f)_* \Big] B^f_{\varepsilon}(v,v_*,\sigma) \, \dd \sigma \, \dd v_*,
\end{equation}
using the same shorthands as in \eqref{eq:BFDoperator}, and
\begin{equation} \label{eqdef:Bfeps}
    B_{\varepsilon}^f(v,v_*,\sigma)  = (1 - \varepsilon f(v)) (1 - \varepsilon f(v_*)) (1 - \varepsilon f(v')) (1 - \varepsilon f(v'_*))\, B(v,v_*,\sigma),
\end{equation}
\noindent The collision kernel $B_{\varepsilon}^f$ straightforwardly satisfies the usual assumptions of symmetry and micro-reversibility which imply the weak form of the classical Boltzmann operator, (see for instance~\cite{CCbook}), so that for any test function $\psi : \R^3 \to \R$,
\begin{equation} \label{eq:weakformBFD}
    \int_{\R^3} \psi(v) \, Q_{\varepsilon}(f,f)(v) \, \dd v = -\frac14 \iiint_{\R^3 \times \R^3 \times \Sb^{2}} \Big[\varphi_{\varepsilon}(f)' \, \varphi_{\varepsilon}(f)'_* - \varphi_{\varepsilon}(f) \varphi_{\varepsilon}(f)_* \Big] \Big( \psi' + \psi'_* - \psi - \psi_* \Big) B^f_{\varepsilon} \dd \sigma \, \dd v \, \dd v_*,
\end{equation}
still using the same shorthands as in \eqref{eq:BFDoperator} for $\varphi_{\varepsilon}(f)$ and $\psi$. Formally taking $\psi(v) = 1$, $\psi(v) = v_i$ for $i = 1,2,3$ and $\psi(v) = \displaystyle \frac12|v|^2$ in \eqref{eq:weakformBFD} yields the conservation of mass, momentum and energy
\begin{equation} \label{eq:conservationmassmomentumenergy}
\int_{\R^3} \begin{pmatrix} 1 \\ v \\ \frac12|v|^2 \end{pmatrix} Q_{\varepsilon}(f,f)(v) \, \dd v = 0.
\end{equation}

\medskip

\noindent In the following, we define, for $p \geq 1$ and $s \in \R$, the weighted $L^p_s$ norm of some function $g$ by
\begin{equation} \label{eqdef:Lpsnorm}
\|g\|_{L^p_s} := \left(\int_{\R^3} |g(v)|^p \, (1 + |v|^2)^{\frac{p s}{2}} \, \dd v \right)^{\frac{1}{p}},
\end{equation}
and its essential supremum by $\|g\|_{\infty}$.
\medskip

\subsection{Macroscopic quantities and equilibrium distribution} Throughout this paper, for any nonnegative $f~\in~L^1_2(\R^3)~\setminus~\{0\}$, we denote $(\rho,u,T) \in \R_+^* \times \R^3 \times \R_+^*$  respectively the density, mean velocity and temperature associated to $f$, in the sense that
\begin{equation} \label{eq:normalizationdef}
\int_{\R^3} f(v) \begin{pmatrix} 1 \\ v \\ |v|^2 \end{pmatrix}  \dd v = \begin{pmatrix} \rho \\ \rho u \\ 3 \rho T + \rho|u|^2 \end{pmatrix}.
\end{equation}
We also define the minimal directional temperature of $f$,
\begin{equation} \label{eqdef:minTi}
    T_*(f) := \frac{1}{\rho} \; \underset{e \in \Sb^2}{\inf} \int_{\R^3} f(v) \, (v \cdot e)^2 \, \dd v,
\end{equation}
as well as the so-called Fermi temperature associated to $\rho$ and $\varepsilon > 0$,
\begin{equation} \label{eq:fermitemp}
    T_F(\rho,\e) := \frac{1}{2} \left( \frac{3 \varepsilon  \rho}{4 \pi}  \right)^{\frac23},
\end{equation}
and the ratio
\begin{equation}
\gamma := \frac{T}{T_F(\rho,\e)},
\end{equation}
with the convention $\gamma = + \infty$ when $\varepsilon = 0$. The adimensional number $\gamma$ is thus the ratio between the actual temperature of the system and its Fermi temperature. In~\cite{lu2001spatially}, Lu proved that
\begin{equation} \label{eq:gammaineq}
1 - \varepsilon f \geq 0 \implies \gamma \geq \frac25,
\end{equation}
as well as, for any fixed $\varepsilon >0$,
\begin{align}
    \left(1 - \varepsilon f \geq 0 \text{ and } \gamma = \frac25 \right) &\iff f \equiv F_{\varepsilon}^{\rho} := \varepsilon^{-1} \, \mathbf{1}_{|\cdot - u| \leq \left( \frac{3 \varepsilon  \rho}{4 \pi}  \right)^{\frac13}}, \label{eq:secondeq}\\
     \left(1 - \varepsilon f \geq 0 \text{ and } \gamma > \frac25 \right) &\iff \text{there exists a unique $\varepsilon$-Fermi-Dirac distribution } \Me \text{ such that } \label{eqdef:fermidiracequilibrium}\\
     &\qquad \int_{\R^3} \Me(v) \begin{pmatrix} 1 \\ v \\ |v|^2 \end{pmatrix}  \dd v = \begin{pmatrix} \rho \\ \rho u \\ 3 \rho T + \rho|u|^2 \end{pmatrix}, \nonumber 
\end{align}
where we recall that the $\varepsilon$-Fermi-Dirac distributions are the ones which are such that $\varphi_{\varepsilon}(\Me)$ is a Maxwellian distribution. Equivalently, $\varepsilon$-Fermi-Dirac distributions are the distributions defined by
\begin{equation} \label{eqdef:FermiDiracStat}
    \M_{\varepsilon}^f(v) = \frac{e^{a_{\e} + b_{\e}|v-u_{\e}|^2}}{1 + \varepsilon e^{a_{\e} + b_{\e}|v-u_{\e}|^2}}, \qquad a_{\e} \in \R, \;  u_{\e} \in \R^3, \; b_{\e} < 0. 
\end{equation}
In particular, a \emph{characterisation} of the Fermi-Dirac (or of the Maxwellian, when $\varepsilon = 0$) distribution $\Me$ is that $\log \circ \, \varphi_{\varepsilon}(\Me)$ is a linear combination of $v \mapsto 1$, $v \mapsto v$ and $v \mapsto |v|^2$ (that is, of the conserved quantities). The interested reader shall find further discussions on this matter in Appendix~\ref{appendix:entropystudy}.

\medskip

\noindent As proven in~\cite{lu2001spatially}, $F^{\rho}_{\varepsilon}$ and $\Me$ are the only two possible equilibria of the (spatially homogeneous) BFD dynamic~\eqref{eq:BFDequation}. Contrary to $\Me$ (called first equilibrium in~\cite{lu2001spatially}), which is an attractor, the distribution $F^{\rho}_{\varepsilon}$ (called second equilibrium in ~\cite{lu2001spatially}) is not, and corresponds to a degenerate state, which can actually occur for very cold gases (note that \eqref{eq:gammaineq}--\eqref{eq:secondeq} imply that $F^{\rho}_{\varepsilon}$ is of ``lowest'' temperature).

\medskip

\noindent From now on, we denote by $M^g$ the Maxwellian distribution associated to a distribution $g$, in the sense that 
$$
M^g \equiv \M_0^g. 
$$

\subsection{Fermi-Dirac entropy} We now define the $\varepsilon$-entropy of the distribution $f$.
\begin{definition}
We define, for any $\varepsilon \geq 0$ and $x \in \R_+$ such that $1 - \varepsilon x \geq 0$,
\begin{equation} \label{eqdef:Phieps}
    \Phi_{\varepsilon}(x) := \int_0^x \, \log \varphi_{\varepsilon} (y) \, \dd y,
\end{equation}
where $\varphi_{\varepsilon}$ is defined in \eqref{eqdef:varphieps}. The function $\Phi_{\varepsilon}$ is well-defined as $\log \in L^1_{loc}(\R_+)$. For any nonnegative $f,g \in L^1_2(\R^3)$ such that $1 - \varepsilon f \geq 0$ and $1 - \varepsilon g \geq 0$, we define the $\varepsilon$-entropy of $f$ by
\begin{equation} \label{eqdef:entropy}
    H_{\varepsilon}(f) := \int_{\R^3} \Phi_{\varepsilon}(f)(v) \, \dd v,
\end{equation}
as well as the relative $\varepsilon$-entropy of $f$ and $g$ by
\begin{equation} \label{eqdef:entropyrelat}
    H_{\varepsilon}[f|g] := H_{\varepsilon}(f) - H_{\varepsilon}(g).
\end{equation}
\end{definition}

\smallskip

\begin{remark} Our definition of entropy slightly differs from the usual ones, but is actually equivalent in the problems that we consider. We choose to define the entropy with \eqref{eqdef:Phieps}--\eqref{eqdef:entropy} because, our study relying on a Taylor expansion, it is here the more natural definition. Let us detail the link between the $\varepsilon$-entropy and usual definitions.

\begin{itemize}
    \item \textbf{Classical entropy.} If $\varepsilon = 0$, then $\Phi_0(x) = x \log x - x$ with the convention $0 \log 0 = 0$, the condition $1 - \varepsilon f \geq 0$ is always satisfied, and
    $$
    H_0(f) = \int_{\R^3} (f \log f - f) \, \dd v.
    $$
    Note that often (consider for instance~\cite{villani2003}), the classical entropy is defined by
    $$
    \widetilde{H}_0(f) = \int_{\R^3} f \log f \, \dd v \equiv \int_{\R^3} \widetilde{\Phi}_0(f) \, \dd v,
    $$
    with $\widetilde{\Phi}_0(x) = x \log x$.
    
    \smallskip
    
    \item \textbf{Fermi-Dirac entropy.} If $\varepsilon > 0$, then $\Phi_{\varepsilon}(x) = x \log x + \varepsilon^{-1} (1-\varepsilon x) \log (1 - \varepsilon x)$ with the convention $0 \log 0 = 0$, and the condition $1 - \varepsilon f \geq 0$ is satisfied if and only if $f \leq \varepsilon^{-1}$. In this case, we have
    $$
    H_{\varepsilon}(f) = \int_{\R^3} f \log f + \varepsilon^{-1} (1 - \varepsilon f) \log (1 - \varepsilon f) \, \dd v.
    $$
    Note that for instance in~\cite{ABDLentropy} (up to a minus sign, which is just a choice of convention), the Fermi-Dirac entropy is defined by
    $$
    \widetilde{H}_{\varepsilon}(f) = \varepsilon^{-1} \int_{\R^3} (\varepsilon f) \log (\varepsilon f) +  (1 - \varepsilon f) \log (1 - \varepsilon f) \, \dd v \equiv \int_{\R^3} \widetilde{\Phi}_{\varepsilon}(f) \, \dd v,
    $$
    with $\widetilde{\Phi}_{\varepsilon}(x) = \varepsilon^{-1} \left(\varepsilon x \log (\varepsilon x) +  (1 - \varepsilon x) \log (1 - \varepsilon x) \right)$.
\end{itemize}

\noindent As we are only interested in the \emph{relative} (Classical of Fermi-Dirac) entropy between $f$ and its associated equilibrium distribution $\Me$ (defined in \eqref{eqdef:fermidiracequilibrium}), we see in the following Proposition~\ref{prop:firstformulaHeps} that only $\Phi_{\varepsilon}''$ is actually relevant to the definition. From $\widetilde{\Phi}_{\varepsilon}'' = \Phi_{\varepsilon}''$, we then obtain $\widetilde{H}_{\varepsilon}[f|\Me] = H_{\varepsilon}[f|\Me]$, ensuring the announced equivalence.
\end{remark}

\begin{proposition} \label{prop:firstformulaHeps}
For any $\varepsilon \geq 0$ and nonnegative $f \in L^1_2(\R^3)\setminus \{0\}$ such that $1 - \varepsilon f \geq 0$ and either $\displaystyle \gamma > \frac25$ if $\e > 0$ or $f \in L \log L (\R^3)$ if $\e=0$, denoting by $\Me$ the Fermi-Dirac (or Maxwellian if $\varepsilon = 0$) distribution associated to $f$, we have the following representation of the relative $\varepsilon$-entropy between $f$ and $\Me$,
\begin{align}
    H_{\varepsilon}[f|\Me]  &= \int_0^1 (1 - \tau) \left[\int_{\R^3} (f - \Me)^2 \; \Phi_{\varepsilon}'' \left((1-\tau) \Me + \tau f \right) \, \dd v \right] \dd \tau \label{eq:firstformulaHeps}\\
    &= \int_{\R^3} \int_{\Me(v)}^{f(v)}  \frac{f(v) - x}{\varphi_{\varepsilon}(x)} \, \varphi_{\varepsilon}'(x) \, \dd x \, \dd v \label{eq:firstformulaHeps2}. 
\end{align}
It holds that
\begin{equation} \label{eq:Meminimizes}
H_{\varepsilon}[f|\Me] \geq 0 \quad \text{ and } \quad (H_{\varepsilon}[f|\Me] = 0 \; \iff \; f = \Me).
\end{equation}
\end{proposition}

\medskip

\noindent The reader will find the proof of \eqref{eq:firstformulaHeps} and \eqref{eq:Meminimizes} in Appendix~\ref{appendix:entropystudy} - see Proposition \ref{prop:general}. We just point out here that this comes straightforwardly from a Taylor expansion of the function $\Phi_{\varepsilon}$, and that \eqref{eq:Meminimizes} is a direct consequence of \eqref{eq:firstformulaHeps} because $\Phi_{\varepsilon}'' > 0$ on $(0,\e^{-1})$. Equation~\eqref{eq:firstformulaHeps2} comes directly from $\Phi''_{\varepsilon} = \frac{\varphi_{\varepsilon}'}{\varphi_{\varepsilon}}$ and the change of variable, for fixed $v$, $x = (1-\tau) \, \Me(v) + \tau \, f(v)$ inside \eqref{eq:firstformulaHeps}.

\begin{remark}
We point out that for any $0< x <\varepsilon^{-1}$, we have
$$
\Phi_{\varepsilon}''(x) = \frac{\varphi_{\varepsilon}'(x)}{\varphi_{\varepsilon}(x)} = \frac{1}{x(1 - \varepsilon x)}.
$$
Hence for any $\tau \in (0,1)$, $(1-\tau) \Me + \tau f \in (0, \varepsilon^{-1})$, and $\Phi_{\varepsilon}'' \left((1-\tau) \Me + \tau f \right)$ is thus well-defined.

\smallskip

\noindent Moreover, with $\varepsilon = 0$ (Classical case), Equations~\eqref{eq:firstformulaHeps}--\eqref{eq:firstformulaHeps2} write
\begin{align}
     H_0 \left[f|M^f \right] &= \int_0^1 (1 - \tau) \left[\int_{\R^3} \frac{(f - M^f)^2}{(1-\tau) M^f + \tau f } \, \dd v \right] \dd \tau \label{eq:firstformulaH02}\\
     &= \int_{\R^3} \int_{M^f(v)}^{f(v)} \frac{f(v)-x}{x} \, \dd x \, \dd v  \label{eq:firstformulaH01}. 
\end{align}
\end{remark}

\medskip

\subsection{Entropy production} We finally define the entropy production $\D_{\varepsilon}$ of the BFD operator by
\begin{equation} \label{eqdef:Deps1}
    \Deps(f) :=  - \int_{\R^3}  \log  \varphi_{\varepsilon}(f) \, Q_{\varepsilon}(f,f)(v) \, \dd v.
\end{equation}
We get, taking $\psi = \log \varphi_{\varepsilon}(f)$ in \eqref{eq:weakformBFD},
\begin{equation} \label{eqdef:Deps2}
 \Deps(f) = \frac14 \iiint_{\R^3 \times \R^3 \times \Sb^{2}} \Big[\varphi_{\varepsilon}(f)' \, \varphi_{\varepsilon}(f)'_* - \varphi_{\varepsilon}(f) \, \varphi_{\varepsilon}(f)_* \Big]  \log \left( \frac{\varphi_{\varepsilon}(f)' \varphi_{\varepsilon}(f)'_*}{\varphi_{\varepsilon}(f) \, \varphi_{\varepsilon}(f)_*} \right)  B^f_{\varepsilon} \, \dd \sigma \, \dd v \, \dd v_* \geq 0.
\end{equation}
Formally, a solution $f$ to the BFD equation \eqref{eq:BFDequation} satisfies
$$
\frac{\dd}{\dd t} H_{\varepsilon}(f(t,\cdot)) = - \Deps(f(t,\cdot)),
$$
or, using the relative entropy,
\begin{equation} \label{eq:entropydissip}
\frac{\dd}{\dd t} H_{\varepsilon}[f(t,\cdot)|\M_{\e}^{f^{in}}] = - \Deps(f(t,\cdot)).
\end{equation}
The nonnegativity of $\Deps$ then ensures the second principle of thermodynamics.
\subsection{Trend to equilibrium and Cercignani's conjecture}
The relative entropy between a solution at a given time $t$ and its associated equilibrium is a very useful quantity to estimate how close to equilibrium the system is, and \eqref{eq:Meminimizes} in Proposition~\ref{prop:firstformulaHeps} makes it a suitable tool into proving convergence towards equilibrium. Indeed, if one shows for some solution $f$ that 
$$
H_{\varepsilon}[f(t,\cdot)|\M_{\e}^{f^{in}}] \underset{t \to \infty}{\longrightarrow} 0,
$$
then the Csiszár-Kullback-Pinsker (CKP) inequalities, discussed below in Proposition~\ref{theorem:CKnonopti} and Section~\ref{appendix:CK} yield
$$
\|f(t,\cdot) - \M_{\e}^{f^{in}}\|_{L^1_2} \underset{t \to \infty}{\longrightarrow} 0,
$$
with a rate of convergence related to the one of the relative entropy. In the study of large-time behaviour of the solutions to the (Classical of Fermi-Dirac) Boltzmann equation, the relative entropy is thus a major tool. The idea of entropy-entropy production methods is to prove an inequality of the type
$$
\Deps(f) \geq \Theta(H_{\varepsilon}[f|\Me]),
$$
for some suitable function $\Theta$ (typically, a linear or power law function), allowing to obtain, combined with \eqref{eq:entropydissip} and using a Gronwall-type argument, an explicit rate of the convergence towards $0$ of the relative entropy - thus of the $L^1_2$ distance, using the CKP inequality.

\bigskip

\noindent In the early 80's, Cercignani conjectured in~\cite{cercioriginal}, \emph{in the context of the Classical Boltzmann equation} with suitable collision kernels, that for any distribution $f$, there exists some $\lambda > 0$ depending only on density, temperature and upper bound on the entropy of $f$ (this is called the strong form of Cercignani's conjecture) such that
$$
\D_0(f) \geq \lambda \, \rho \, H_0[f|M^f],
$$
where we recall that $\rho$ is the density associated to $f$, $\D_0$ is the entropy production in the classical case, $H_0$ the classical entropy and $M^f$ the Maxwellian associated to $f$. It is now known that this is false in general (see the negative result of Bobylev and Cercignani~\cite{bobylev1999}), however Villani proved in~\cite{villani2003} the strong form of this conjecture when the collision kernel $B$ is "super-quadratic", as well as a weak form $\D_0(f) \geq C(f) \, H_0[f|M^f]^{1 + \delta}$, holding for all $\delta > 0$, for general collision kernels (and $C(f)$ depending on various quantities related to $f$). We mention the cornerstone papers of Toscani and Villani \cite{toscanivillani99,toscanivillani00} on this topic, and refer the reader to~\cite{devillettes2011} for a review of Cercignani's conjecture results for the Boltzmann and Landau equations.

\medskip

\noindent Similar ideas hold in the study of the Landau equation (we shall talk about Cercignani's conjecture for the Landau equation), this topic is further addressed in Section~\ref{section:LFD}.

\medskip

\noindent \textbf{The motivation of this work is to provide Cercignani's conjecture type results in the case of the Boltzmann-Fermi-Dirac equation, uniformly with respect to $\varepsilon > 0$.}

\medskip

\noindent The simple yet fundamental idea which is at the root of this work is the fact that, whenever $\varepsilon > 0$, assuming $1 - \varepsilon f \geq \kappa_0$ for some $\kappa_0 \in (0,1)$ straightforwardly implies
\[
B(v,v_*,\sigma) \geq B^f_{\varepsilon}(v,v_*,\sigma) \geq \kappa_0^4 \, B(v,v_*,\sigma),
\]
so that, as $\log$ in nonincreasing,

\begin{equation} \label{eq:compareproductions}
\D_0 \left( \varphi_{\varepsilon}(f) \right) \geq \Deps(f) \geq \kappa^4_0 \, \D_0 \left( \varphi_{\varepsilon}(f) \right).
\end{equation}

\noindent Therefore, relatively to $\kappa_0$, the quantity $\Deps(f)$ (the Fermi-Dirac entropy production associated to $f$) behaves like $\D_0\left( \varphi_{\varepsilon}(f) \right)$ (the Classical entropy production associated to $\varphi_{\varepsilon}(f)$). Our main results, Theorem~\ref{theorem:FD} and Proposition~\ref{prop:FDupper}, allow to state a similar property for the relative entropies. We find that, relatively to $\kappa_0$, the quantity $H_{\varepsilon}[f|\Me]$ (the Fermi-Dirac relative-entropy between $f$ and $\Me$) behaves like $H_0[\varphi_{\varepsilon}(f)|M^{\varphi_{\varepsilon}(f)}]$ (the Classical relative entropy between $\varphi_{\varepsilon}(f)$ and $M^{\varphi_{\varepsilon}(f)}$). This entails that \textbf{any entropy-entropy production inequality in the Classical case yields the same inequality} (with an almost identical constant) \textbf{in the Fermi-Dirac case}; this corresponds to Corollary~\ref{theorem:BFDcerci} (and Proposition~\ref{theorem:LFDcerci} for the Landau/Landau-Fermi-Dirac adaptation). The same holds for counter-examples to entropy-entropy production inequalities.

\section{Main results}  \label{section:mainresults}

\noindent In this section, we state our main results.

\subsection{Comparison of the Fermi-Dirac and Classical relative entropies} Theorem~\ref{theorem:FD} is the main contribution of this work, as it establishes a strong link between the different entropies, allowing to transfer ``positive'' Cercignani's conjecture results from the classical to the fermionic case.

\begin{theorem} \label{theorem:FD} 
\textbf{Lower-bound inequality.} For any $\varepsilon > 0$ and nonnegative $f \in L^1_2(\R^3) \setminus \{0\}$ such that
$$
1 - \varepsilon f > 0 \quad \text{and} \qquad \frac{f}{1-\varepsilon f} \in L^1_2(\R^3) \cap L \log L (\R^3),
$$
we have
\begin{equation}\label{eq:theorementropiesFD2}
 H_0\left[ \left. \frac{f}{1 - \varepsilon f} \right| M^{\frac{f}{1 - \varepsilon f}} \right] \geq  H_{\varepsilon}[f|\Me].
\end{equation}
\end{theorem}

\noindent The hypotheses of Theorem~\ref{theorem:FD} are actually the minimal ones to ensure that every term in~\eqref{eq:theorementropiesFD2} makes sense. The following Proposition~\ref{prop:FDupper} allows to transfer ``negative'' Cercignani's conjecture results from the classical to the fermionic case.

\begin{proposition} \label{prop:FDupper}
\textbf{Upper-bound inequality.} For any $\kappa_0 \in (0,1)$, any $\varepsilon >0$ and nonnegative $f \in L^1_2(\R^3) \setminus \{0\}$ such that
\begin{equation} \label{eqassump:kappa02}
1 - \varepsilon f \geq \kappa_0,
\end{equation}
we have
\begin{equation}\label{eq:theorementropiesFD}
 H_0\left[ \left. \frac{f}{1 - \varepsilon f} \right| M^{\frac{f}{1 - \varepsilon f}} \right] \leq C_0(\kappa_0) \, H_{\varepsilon}[f|\Me],
\end{equation}
with $C_0(\kappa_0) = \exp \left(16 \, (\kappa_0^{-1}-1) \right)$.
\end{proposition}

\medskip

\begin{remark}
Inequality~\eqref{eq:theorementropiesFD2} becomes reversed in the Bose-Einstein case. We refer the interested reader to Proposition~\ref{prop:BE} in Appendix~\ref{appendix:boseeinstein}. In fact, we expect that both inequalities, but reversed (and probably with a different constant $C_0$), hold in the Bose-Einstein case, as the core of the proof would be the exact same in this case, only changing $\e$ into $-\e$ (hence the reversal of the inequalities).
\end{remark}

\medskip

\subsection{Entropy-entropy production inequalities for the Boltzmann-Fermi-Dirac (BFD) equation} As a corollary to Theorem~\ref{theorem:FD} and Proposition~\ref{prop:FDupper}, adapting the main results found in Villani~\cite{villani2003}, we obtain the following entropy-entropy production estimates.
 \begin{corollary}
 \label{theorem:BFDcerci}

$\bullet$ \textbf{Cercignani's conjecture result for the BFD operator with super-quadratic kernels} \emph{(adaptation of~\cite[Theorem 2.1]{villani2003})}. Assume that $B$ satisfies \eqref{eqassump:B} and
$$
B(v,v_*,\sigma) \geq 1 + |v-v_*|^2.
$$

\smallskip

\noindent Then, for any $\kappa_0 \in (0,1)$, $\varepsilon \geq 0$ and nonnegative $f \in L^1_2(\R^3) \setminus \{0\}$ such that   $1 - \varepsilon f \geq \kappa_0$, we have

\smallskip

\begin{equation} \label{eq:cercignanisuperquadra}
    \Deps (f)  \geq K(f) \, H_{\varepsilon}[f|\Me],
\end{equation}
with
\begin{equation} \label{eq:Ksuperquadra}
    K(f) =  \frac{2\pi}{7} \kappa_0^5  \min (1, \, T)\, \rho \, T_{*}(f),
\end{equation}  
where $\rho$, $T$, $T_*(f)$ are defined in \eqref{eq:normalizationdef}--\eqref{eqdef:minTi}, $\Deps$, $H_{\varepsilon}$ and $\Me$ are given respectively in \eqref{eqdef:Deps1}--\eqref{eqdef:Deps2}, \eqref{eqdef:entropy}--\eqref{eqdef:entropyrelat} and \eqref{eqdef:fermidiracequilibrium} (note that $T_*(f)$ can be bounded below using the entropy or an $L^{\infty}$ bound~\cite[Proposition 2]{desvillettesvillani2}).
\medskip

$\bullet$ \textbf{Super-linear Cercignani's conjecture result for the BFD operator with general kernels} \emph{(adaptation of~\cite[Theorem 4.1]{villani2003})}. Assume the existence of $B_0 > 0$ and $\beta_+,\beta_- \geq 0$ such that
\begin{equation} 
B(v,v_*,\sigma) \geq B_0 \, \min \left( |v-v_*|^{\beta_+}, \, |v-v_*|^{-\beta_-} \right),
\end{equation}
and $B$ satisfies \eqref{eqassump:B}. We consider $\kappa_0 \in (0,1)$,  $0 \leq f\in L^1_2(\R^3) \setminus \{0\}$ and $\varepsilon \geq 0$ such that {  $1 - \varepsilon f \geq \kappa_0$} and
\begin{equation} \label{eqassump:lowermaxwell}
    \forall \, v \in \R^3, \qquad f(v) \geq K_0 \, e^{-A_0 |v|^{q_0}} \qquad \qquad (K_0 > 0, \quad A_0 > 0, \quad q_0 \geq 2).
\end{equation}
Then for any $\alpha \in (0,1)$ there exists an explicit constant $K_{\alpha}(f)$ depending on $\alpha$, $B_0$, an upper bound of $\varepsilon$ and on $f$ only through $\rho$, $T$, $q_0$, and upper bounds on $A_0$, $1/K_0$, $1 / \kappa_0$, $\|f\|_{L^1_s}$ and $\|f\|_{H^k}$, where $s=s(\alpha,q_0,\beta_-,\beta_+)$ and $k=k(\alpha,s,\beta_-,\beta_+)$, such that
\begin{equation} \label{eq:ineqentropygeneralkernel}
    \Deps (f)  \geq K_{\alpha}(f) \, H_{\varepsilon}[f|\Me]^{1+ \alpha}.
\end{equation}

\medskip

$\bullet$ \textbf{Counter-example to Cercignani's conjecture for the BFD operator with sub-quadratic kernels} \emph{(weak adaptation of~\cite{bobylev1999} and~\cite[Theorem 1.1]{villani2003})}. Assume that $B$ satisfies \eqref{eqassump:B} and that there exists $B_0 > 0$ and $0 \leq \beta < 2$ such that
$$
\int_{\Sb^2} B(v,v_*,\sigma) \, \dd \sigma \leq B_0 \, (1 + |v-v_*|^{\beta}).
$$
Then for any  $\varepsilon \in (0, \frac12)$, we have
$$
\inf \left\{ \left. \frac{\Deps(f)}{H_{\varepsilon}[f|\Me]} \right| \, f \in \mathcal{C}^{\varepsilon}_{1,0,1} \emph{ and }  \, f(v) \geq \frac12 (2 \pi)^{-3/2} \, e^{- |v|^2}, \; \;v \in \R^3 \right\} = 0,
$$
with
$$
f \in \mathcal{C}^{\varepsilon}_{1,0,1} \iff 0 \leq f \in L^1_2(\R^3), \quad 1 - \varepsilon f \geq 0, \quad 1 \leq \rho \leq 1 + 2 \varepsilon, \quad |u| \leq 4 \sqrt{3} \, \varepsilon, \quad  \frac{1}{1 + 2 \varepsilon} \leq T \leq 1 + 2 \varepsilon.
$$

\end{corollary}

\noindent We recall that $\rho,u$ and $T$ are defined in~\eqref{eq:normalizationdef}. The last part of the corollary (the counter-example) is a weak version of~\cite[Theorem 1]{bobylev1999}, due to the presence of $\mathcal{C}^{\varepsilon}_{1,0,1}$ instead of the more expected $\mathcal{C}^{0}_{1,0,1}$ (the set of normalized functions) and the lack of a condition on the $\|\cdot\|_{L^1_s}$ and $\|\cdot\|_{H^k}$ norms. One could probably replace $\mathcal{C}^{\varepsilon}_{1,0,1}$ by $\mathcal{C}^0_{1,0,1}$ and add the $\| \cdot \|_{L^1_s}$ and $\| \cdot \|_{H^k}$ conditions with further technical work.

\subsection{About the Csiszár-Kullback-Pinsker inequality} \label{subsection:CKP}
\noindent The Csiszár-Kullback-Pinsker (CKP) inequality is a well-known link between the total variation distance of two probability measures and their relative entropy, which is often used to transform a convergence in relative entropy into a convergence in an actual norm, such as $L^1$ or $L^1_2$. In conducting our study, we were led to prove Proposition~\ref{theorem:CKnonopti} (or the optimal Proposition~\ref{theorem:CK}), which are generalizations of the CKP inequality.

\medskip

\noindent A generalization of the CKP inequality allowing for a weight in the total variation distance, called "weighted Csiszár-Kullback-Pinsker inequalities" was already provided by Bolley and Villani in~\cite[Theorem 2.1]{bolleyvillani}. However our inequality does differ from theirs. On the one hand, as far as drawbacks are concerned, we need the reference measure to be an equilibrium distribution while it is not the case in~\cite{bolleyvillani}, and our constant depends on $f$ while it doesn't in the inequality in~\cite{bolleyvillani}. On the other hand, for the improvements, our constant seems more precise and a great deal smaller (possibly the constant in Proposition~\ref{theorem:CK} is optimal), in the sense that choosing a weight equal to $|v|^{\alpha}$ with $\alpha > 2$ yields a finite constant in our case but an infinite one in~\cite[Theorem 2.1]{bolleyvillani}. In particular, remark that by Jensen's inequality, $\|M \, \varpi^2\|_{L^1} \leq \log \int e^{\varpi^2} \, M \, \dd v$.

\medskip

\noindent In another direction, in the study of the Vlasov-Poisson system, Cáceres et al. provided a generalization of the CKP inequality to $L^p$ norms with $1 \leq p \leq 2$ in~\cite[Proposition 3.1]{caceres2002nonlinear} (specifically in the note~\cite{notecarillo}) for various different entropies, which formulation and proof is quite close to our result. However, as far as the standard entropy is concerned, their result works only for $p=1$, is without weights, and their constant is not optimal. I am not aware of other such generalizations of the CKP inequality in the literature.

\medskip

\noindent Finally, the main idea of our proof of this inequality was already present in Jüngel's~\cite[proof of Theorem A.2]{jungel2016entropy}, and he also proposed a generalization in~\cite[Theorem A.3]{jungel2016entropy}, which differs from ours and holds for the $L^1$ norm.

\medskip

\noindent In the following Proposition~\ref{theorem:CKnonopti} (and in Proposition~\ref{theorem:CK} in Section~\ref{appendix:CK}), we propose an extension the CKP inequality to weighted $L^p$ distances for $p \in [1,2]$, as well as to other entropies than the Classical entropy (typically, the Fermi-Dirac and the Bose-Einstein entropies), which could be seen as a extension of the previously mentionned three generalizations. We denote this extension by $L^1$-$L^2$ weighted Csiszár-Kullback-Pinsker inequality.

\medskip

\noindent We moreover wish to highlight the fact that the proof of these inequalities, which is the same as in~\cite[Theorem A.2]{jungel2016entropy}, but with the addition of weights and of considering $L^p$, $1 \leq p \leq 2$, relies on a Hölder inequality. In particular, our proof of the original CKP inequality relies on a Taylor expansion and a Cauchy-Schwarz argument, instead of the tricky inequality (see for instance~\cite{gilardoni}) $\displaystyle x \log x - x + 1 \geq \frac32 \times \frac{(x-1)^2}{x+2}$.

\medskip

\noindent For the sake of readability, we present hereafter a simplified non-optimal version of the proposition, and refer the reader to Section~\ref{appendix:CK} for the refined and more general version of it, where it is also extended to the Bose-Einstein case and further discussed. We also extend it to a broader class of general entropies ($\Phi$-entropies in the text) in Proposition~\ref{theorem:generalCK} of Appendix~\ref{appendix:entropystudy}.

\begin{proposition} \label{theorem:CKnonopti} \textbf{$L^1$-$L^2$ weighted Csiszár-Kullback-Pinsker inequality related to the Classical and Fermi-Dirac relative entropies. \emph{[simplified]}}

\smallskip

\noindent Let $\varpi : \R^3 \to \R_+$ be measurable, and $r \in [1,2]$.
\medskip

$\bullet$ \textbf{Classical CKP inequality.} For any $0\leq f \in L^1_2(\R^3) \cap L \log L (\R^3) \setminus \{0\}$, assuming that the norms below are finite, it holds
\begin{equation} \label{eq:CKL1L2classicalnonopti}
\left\|(f-M^f) \, \varpi \right\|^2_{L^r} \leq 2 \, \max \left(\left\|M^f \, \varpi^2 \right\|_{L^{\frac{r}{2-r}}}, \, \left\|f\, \varpi^2 \right\|_{L^{\frac{r}{2-r}}} \right)  H_{0}[f|M^f],
\end{equation}
where $M^f$ is the Maxwellian associated to $f$.

\medskip

$\bullet$ \textbf{Fermi-Dirac CKP inequality.} For any $\varepsilon > 0$ and $0\leq f \in L^1_2(\R^3) \setminus \{0\}$ such that $1 - \varepsilon f \geq 0$ and $\gamma > \frac25$, assuming that the norms below are finite,
\begin{equation} \label{eq:CKL1L2ferminonopti}
\left\|(f-\Me) \, \varpi  \right\|^2_{L^r} \leq 2 \, \max \left(\left\|\Me \, \varpi^2  \right\|_{L^{\frac{r}{2-r}}}, \, \left\|f \, \varpi^2  \right\|_{L^{\frac{r}{2-r}}} \right)   H_{\varepsilon}[f|\Me],
\end{equation}
where $\Me$ is the $\varepsilon$-Fermi-Dirac distribution associated to $f$.
\end{proposition}
\noindent   In the above proposition,  when $r=2$, $L^{\frac{r}{2-r}}$ should be understood as  $L^{\infty}$.

\subsection{Convergence towards equilibrium}
A result like Corollary~\ref{theorem:BFDcerci} constitutes the core of the proof of convergence of solutions to the Boltzmann-Fermi-Dirac equation towards equilibrium based on an entropy method. The proof is then complete provided that the solution satisfies a variety of bounds. The obtention of these bounds is an ongoing work that Bertrand Lods and I are currently conducting in the cutoff hard potentials case.

\begin{proposition} \label{prop:ifthmBFD}
    Let $\e > 0$ and $f^{in} \in L^1_2(\R^3) \setminus \{0\}$ such that $0 \leq f^{in} \leq \e^{-1}$. We consider a collision kernel $B$ satisfying the symmetry and micro-reversibility assumptions, thus having the form~\eqref{eqassump:B}, as well as, for some $B_0 >0$ and $\beta^+, \beta^- \geq 0$,
    $$
    B(|v-v_*|,\cos \theta) \geq B_0 \, \min \left( |v-v_*|^{\beta_+}, \, |v-v_*|^{-\beta_-} \right).
    $$
    We assume the existence of a solution $f \equiv f(t,v)$ to the spatially homogeneous Boltzmann-Fermi-Dirac equation~\eqref{eq:BFDequation} associated to the quantum parameter $\e$, the collision kernel $B$ and the initial distribution $f^{in}$. Assume that there exists an explicit time $t_0 > 0$ depending only on $\e$, $B$ and $f^{in}$ such that
    \begin{itemize}
        \item $\exists \, \kappa_0 \in (0,1)$ depending only on $\e$, $B$ and $f^{in}$ such that $1- \e f(t,v) \geq \kappa_0$ for any $(t,v) \in [t_0,+\infty) \times \R^3$,
        \item $\exists \, K_0,A_0 > 0$ depending only on $\e$, $B$ and $f^{in}$ such that $f(t,v) \geq K_0 e^{-A_0 |v|^2}$ for any $(t,v) \in [t_0,+\infty) \times \R^3$,
        \item $\forall \, k,l \in \N$, there exist $C_k > 0,\widetilde{C}_l > 0$ depending only on $\e$, $B$ and $f^{in}$ such that $\sup_{t \geq t_0} \|f(t,\cdot)\|_{L^1_k} \leq C_k$ and $\sup_{t \geq t_0} \|f(t,\cdot)\|_{H^l} \leq \widetilde{C}_l$.
    \end{itemize}
\noindent Then for any $p \geq 1$, $l \geq 0$ and $\alpha > 0$ there exists an explicit constant $\mathbf{C}_{p,l, \alpha} > 0$ depending only on $p,l$, $\e$, $B$ and $f^{in}$ such that for any $t \geq t_0$,
\begin{equation} \label{eqthmcvBFDafterCK}
    \left\|f(t,\cdot) - \M_{\e}^{f^{in}} \right\|_{L^p_l} \leq \mathbf{C}_{p,l, \alpha} \, (1+t)^{-1/\alpha}.
\end{equation}
\end{proposition}

\begin{proof}
The assumptions on the solution $f$ allow to apply Corollary~\ref{theorem:BFDcerci}, more specifically Equation~\eqref{eq:ineqentropygeneralkernel}, yielding the fact that for any $t \geq t_0$ and $\alpha > 0$, there is an explicit constant $K_{\alpha}(f(t,\cdot))$ depending only on $\alpha$, $B_0$, an upper bound of $\varepsilon$ and on $f(t,\cdot)$ only through $\rho$, $T$ (its density and temperature, defined by~\eqref{eq:normalizationdef}), and upper bounds on $A_0$, $1/K_0$, $1 / \kappa_0$, $\|f(t,\cdot)\|_{L^1_s}$ and $\|f(t,\cdot)\|_{H^k}$, where $s=s(\alpha,q_0,\beta_-,\beta_+)$ and $k=k(\alpha,s,\beta_-,\beta_+)$, such that
\begin{equation} \label{eqproofcvBFD}
    \Deps (f(t,\cdot))  \geq K_{\alpha}(f(t,\cdot)) \, H_{\varepsilon}[f(t,\cdot)|\M_{\e}^{f(t,\cdot)}]^{1+ \alpha}.
\end{equation}
From our assumptions, we observe that $K_{\alpha}(f(t,\cdot))$ can be upper-bounded by a constant $\hat{K}_{\alpha}$ depending only on $\e$, $B$ and $f^{in}$ (in particular not on time $t$). Moreover, since $f$ solves~\eqref{eq:BFDequation}, we have for all $t \geq 0$ that $\M_{\e}^{f(t,\cdot)} = \M_{\e}^{f^{in}}$ and, at least formally, that
$$
\frac{\dd }{\dd t} H_{\varepsilon}[f(t,\cdot)|\M_{\e}^{f^{in}}] = - \Deps (f(t,\cdot)).
$$
Then~\eqref{eqproofcvBFD} implies
$$
\frac{\dd }{\dd t} H_{\varepsilon}[f(t,\cdot)|\M_{\e}^{f^{in}}] \leq - \hat{K}_{\alpha} \, H_{\varepsilon}[f(t,\cdot)|\M_{\e}^{f^{in}}]^{1+ \alpha}.
$$
Integrating this equation between $t_0$ and $t > t_0$ yields
\begin{equation*}
H_{\varepsilon}[f(t,\cdot)|\M_{\e}^{f^{in}}] \leq \left( H_{\varepsilon}[f(t_0,\cdot)|\M_{\e}^{f^{in}}]^{\alpha} + \alpha \hat{K}_{\alpha} (t-t_0) \right)^{-1/\alpha}.
\end{equation*}
We then conclude that for any $t \geq t_0$, (even if $H_{\varepsilon}[f(t_0,\cdot)|\M_{\e}^{f^{in}}]=0$, in which case $H_{\varepsilon}[f(t,\cdot)|\M_{\e}^{f^{in}}]=0$)
\begin{equation} \label{entropyineqBFDcv}
H_{\varepsilon}[f(t,\cdot)|\M_{\e}^{f^{in}}]  \leq \hat{C}_{\alpha} \, (1 + t)^{-1/\alpha},
\end{equation}
where $\hat{C}_{\alpha} > 0$ is explicit and depends only on $\alpha$, $\e$, $B$ and $f^{in}$. Finally, we recall the CKP inequalities provided in Proposition~\ref{theorem:CKnonopti}, Equation~\eqref{eq:CKL1L2ferminonopti} in the Fermi-Dirac context, for weighted $L^1$ norms, implying that for any $k \geq 0$ we have for any $t \geq t_0$
$$
\left\|f(t,\cdot)-\M_{\e}^{f^{in}} \right\|^2_{L^1_k} \leq 2 \, \max \left(\left\|\M_{\e}^{f^{in}} \right\|_{L^1_{2k}}, \, \left\|f(t,\cdot) \right\|_{L^1_{2k}} \right)   H_{\varepsilon}[f(t,\cdot)|\M_{\e}^{f^{in}}].
$$
Since by assumption $\left\|f(t,\cdot) \right\|_{L^1_{2k}} \leq C_{2k}$, we conclude that there exists an explicit $C'_k$ depending only on $\e$, $B$ and $f^{in}$ such that
$$
\left\|f(t,\cdot)-\M_{\e}^{f^{in}} \right\|^2_{L^1_k} \leq C'_k \, H_{\varepsilon}[f(t,\cdot)|\M_{\e}^{f^{in}}] \leq C'_k \, \hat{C}_{\alpha} \, (1 + t)^{-1/\alpha}.
$$
Then~\eqref{eqthmcvBFDafterCK} for the $L^1_k$ norm is finally obtained with $\mathbf{C}_{1,k,\alpha} = \sqrt{C'_k \, \hat{C}_{\alpha/2}}$ and $L^p_l$ norms follow, as $\M_{\e}^{f^{in}} \in L^{\infty}(\R^3)$ and $f \in L^{\infty}(\R_+ \times \R^3)$.
\end{proof}

\section{Proof of Theorem~\ref{theorem:FD} and Proposition~\ref{prop:FDupper}} \label{section:prooftheoremFD}  
\noindent This section is devoted to the proof of Theorem~\ref{theorem:FD} and Proposition~\ref{prop:FDupper}. We make use of Proposition~\ref{prop:firstformulaHeps}, which is later proven in Appendix~\ref{appendix:entropystudy} - see Proposition \ref{prop:general}.

\subsection{Key element of the proof} The main argument in the proofs of Theorem~\ref{theorem:FD} and Proposition~\ref{prop:FDupper} lies in the following proposition.

\begin{proposition} \label{prop:rgnonincr}
    Let $0 \leq g \in L^1_2(\R^3) \cap L \log L(\R^3)  \setminus \{0\}$ and define, for $\varepsilon \geq 0$,
\begin{equation}  \label{eqdef:reps}
 R_g(\varepsilon) := H_{\varepsilon} \left[\varphi_{\varepsilon}^{-1}(g) \left|\M_{\varepsilon}^{\varphi_{\varepsilon}^{-1}(g)} \right. \right].
\end{equation}
Then $\varepsilon \mapsto R_g(\varepsilon)$ is $\mathcal{C}^1$ on $\R_+^*$, with, for any $\varepsilon > 0$, 
\begin{equation} \label{eq:Rgprime}
R'_g(\e) = \int_{\R^3} \int_{M_{\e}(v)}^{g(v)} \frac{\partial_{\e}(\vi_{\e}^{-1})(g(v)) - \partial_{\e}(\vi_{\e}^{-1})(y)}{y} \, \dd y \, \dd v,
\end{equation}
where we used the shorthand $M_{\e} = \vi_{\e}(\M_{\e}^{\vi_{\e}^{-1}(g)})$.
\end{proposition}

\noindent The striking point of~\eqref{eq:Rgprime} lies in the fact that it immediately proves that, for any $\e> 0$, the global monotonicity of $x \mapsto \partial_{\e} (\varphi^{-1}_{\e})(x)$ translates into the same monotonicity of $R_g$ at the point $\e$, as the integral in $y$ in~\eqref{eq:Rgprime} is automatically nonnegative if $x \mapsto \partial_{\e} (\varphi^{-1}_{\e})(x)$ is nondecreasing, or nonpositive if $x \mapsto \partial_{\e} (\varphi^{-1}_{\e})(x)$ is nonincreasing. Moreover, this proposition generalises to any family of entropies and any conserved quantities (provided a suitable differentiability property for $R_g$), as all key arguments rely only the generic framework we present in Appendix~\ref{appendix:entropystudy}, and the specificity of the Fermi-Dirac case is used only for the study of the differentiability of $R_g$.

\begin{proof}
    Using the relative entropy representation~\eqref{eq:firstformulaHeps2} in Proposition~\ref{prop:firstformulaHeps}, we have, for any $\e > 0$,
    \smallskip
    $$
    R_g(\e) = \int_{\R^3} \int_{\M_{\e}^{\vi_{\e}^{-1}(g)}(v)}^{\vi_{\e}^{-1}(g(v))} (\vi_{\e}^{-1}(g(v)) - x)\frac{\vi_{\e}'(x)}{\vi_{\e}(x)} \, \dd x \, \dd v.
    $$
    \smallskip
    Changing variables $y = \vi_{\e}(x)$, with then $\dd y = \vi_{\e}'(x) \, \dd x$ and $x = \varphi_{\e}^{-1}(y)$, we obtain
\smallskip
    $$
    R_g(\e) = \int_{\R^3} \int_{M_{\e}(v)}^{g(v)} \frac{\vi_{\e}^{-1}(g(v)) - \vi_{\e}^{-1}(y)}{y} \, \dd y \, \dd v.
    $$
where we denoted $M_{\e} = \vi_{\e}(\M_{\e}^{\vi_{\e}^{-1}(g)})$ for convenience of notations. We fix $v \in \R^3$ and set, for $(\e,z) \in (\R_+^*)^2$,
\begin{equation} \label{eqdef:K2var}
K(\e,z) := \int_{z}^{g(v)} \frac{\vi_{\e}^{-1}(g(v)) - \vi_{\e}^{-1}(y)}{y} \, \dd y.
\end{equation}
Note that if $g(v) = 0$ then $K(\e,z) = \e^{-1} \log (1 + \e z)$. The application $K$ is clearly $\mathcal{C}^1$ on $(\R_+^*)^2$, and straightforward computations can provide that differentiation in $\e$ under the integral in $y$ is valid. We find that
\begin{align*}
\partial_{\e} K(\e,z) &=  \int_{z}^{g(v)}  \frac{\partial_{\e}(\vi_{\e}^{-1}(g(v))) - \partial_{\e} (\vi_{\e}^{-1}(y))}{y}   \, \dd y,  \\
\partial_{z} K(\e,z) &= - \frac{\vi_{\e}^{-1}(g(v)) - \vi_{\e}^{-1}(z)}{z}.
\end{align*}
As a result of Lemma~\ref{lemma:derivMeps} in Appendix~\ref{appendix:technicallemmas}, the application $\e \mapsto M_{\e}(v)$ is $\mathcal{C}^1$ on $\R_+^*$ (as a composition of $\mathcal{C}^1$ applications), hence so is $\e \mapsto K(\e,M_{\e}(v))$, and we have, for any $\e > 0$ and $v \in \R^3$,
\begin{equation} \label{eq:derivativeKeps}
    \frac{\dd}{\dd \e}  K(\e,M_{\e}(v)) = \int_{M_{\e}(v)}^{g(v)}  \frac{\partial_{\e}(\vi_{\e}^{-1})(g(v)) - \partial_{\e}(\vi_{\e}^{-1})(y)}{y}   \, \dd y - \left(\partial_{\e} M_{\e}(v) \right) \, \frac{\vi_{\e}^{-1}(g(v)) -   \vi_{\e}^{-1}(M_{\e}(v))}{M_{\e}(v)}.
\end{equation}
Now note that
\begin{equation} \label{eq:rgkeps}
R_g(\e) = \int_{\R^3} K(\e,M_{\e}(v)) \, \dd v.
\end{equation}
Then $R_g$ is indeed $\mathcal{C}^1$ on $\R_+^*$, and differentiation under the integral is valid, if we can prove that for any $\e > 0$, there is a $\delta \in (0,\e)$ such that
\begin{equation} \label{eqproofsupK}
\int_{\R^3} \sup_{\e_* \in (\e-\delta,\e+\delta) }\left|  \frac{\dd}{\dd \e_*} \, K(\e_*,M_{\e_*}(v)) \right| \dd v < \infty.
\end{equation}
For the sake of clarity, we continue our computations assuming~\eqref{eqproofsupK} for now, and show~\eqref{eqproofsupK} after we have obtained our claimed result~\eqref{eq:Rgprime}.  

\medskip

\noindent Differentiating~\eqref{eq:rgkeps} in $\e$ under the integral in $v$, using~\eqref{eq:derivativeKeps}, yields for any $\e > 0$ (omitting to write the dependence in $v$ for clarity) 
\begin{equation} \label{eq:firstderivRproofTh1}
R'_g(\e) = \int_{\R^3} \int_{M_{\e}}^{g} \frac{\partial_{\e}(\vi_{\e}^{-1})(g) -\partial_{\e}(\vi_{\e}^{-1})(y)}{y} \, \dd y \, \dd v - \int_{\R^3} (\partial_{\e}M_{\e}) \times \frac{\vi_{\e}^{-1}(g) - \vi_{\e}^{-1}(M_{\e})}{M_{\e}} \, \dd v.
\end{equation}
Now remark that
\begin{equation} \label{eq:Mepslogderiveps}
\frac{\partial_{\e} M_{\e}}{M_{\e}} = \partial_{\e} \left(\log \circ \, \varphi_{\e} \left(\M_{\e}^{\varphi_{\e}^{-1}(g)} \right) \right).
\end{equation}
Since $\M_{\e}^{\varphi_{\e}^{-1}(g)}$ is a Fermi-Dirac statistics, we have, denoting its coefficients by $a_{\e}$,  $u_{\e}$ and $b_{\e}$ (as in~\eqref{eqdef:FermiDiracStat}), that for any $v \in \R^3$,
\begin{equation} \label{eq:logMepslinearcomb}
  \log \circ \, \varphi_{\e} \left(\M_{\e}^{\varphi_{\e}^{-1}(g)} \right) (v)  = \alpha_{\e} + \beta_{\e} \cdot v + b_{\e} \, |v|^2,
\end{equation}
where $\alpha_{\e} =  a_{\e} + b_{\e} |u_{\e}|^2 $ and $\beta_{\e} = - 2b_{\e} \, u_{\e}$. Lemma~\ref{lemma:derivMeps} in Appendix~\ref{appendix:technicallemmas} ensures that the application $\varepsilon \mapsto (a_{\e}, u_{\e}, b_{\e})$ is $\mathcal{C}^1$ on $\R_+^*$, hence so is the above expression pointwise in $v$, and we conclude that $\displaystyle \frac{\partial_{\e} M_{\e}}{M_{\e}}$ is a linear combination of the functions $v \mapsto 1$, $v \mapsto v$ and $v \mapsto |v|^2$. By construction, $\vi_{\e}^{-1}(g)$ and $ \vi_{\e}^{-1}(M_{\e}) = \M_{\e}^{\varphi_{\e}^{-1}(g)}$ share the same normalization in $v \mapsto 1$, $v \mapsto v$ and $v \mapsto |v|^2$, hence
$$
\int_{\R^3} \left(\vi_{\e}^{-1}(g) - \vi_{\e}^{-1}(M_{\e}) \right) \frac{\partial_{\e}M_{\e}}{M_{\e}} \, \dd v = 0,
$$
and~\eqref{eq:firstderivRproofTh1} becomes the announced~\eqref{eq:Rgprime}. 

\bigskip

\noindent We now show~\eqref{eqproofsupK}, ensuring the validity of the differentiation in $\e$ under the integral in Equation~\eqref{eq:rgkeps}, by dominated convergence. The previous computations imply that for any $\e > 0$,
\begin{align*}
\left| \frac{\dd}{\dd \e} \, K(\e,M_{\e}(v)) \right| &\leq \left|\int_{M_{\e}}^{g} \frac{\partial_{\e}(\vi_{\e}^{-1})(g) -\partial_{\e}(\vi_{\e}^{-1})(y)}{y} \, \dd y \right| + \left|  \frac{\partial_{\e}M_{\e} }{M_{\e}} (\vi_{\e}^{-1}(g) - \vi_{\e}^{-1}(M_{\e}))\right|.
\end{align*}
On the one hand, using~\eqref{eq:Mepslogderiveps}--\eqref{eq:logMepslinearcomb}, we get, for any $v \in \R^3$,
\begin{align*}
\left|  \frac{\partial_{\e}M_{\e} }{M_{\e}} (\vi_{\e}^{-1}(g) - \vi_{\e}^{-1}(M_{\e}))\right|  &\leq \left( |\partial_{\e} \alpha_{\e}| + |\partial_{\e} \beta_{\e}| + |\partial_{\e} b_{\e}| \right) (1 + |v|^2) \, |\vi_{\e}^{-1}(g(v)) - \vi_{\e}^{-1}(M_{\e}(v))| \\
&\leq \left( |\partial_{\e} \alpha_{\e}| + |\partial_{\e} \beta_{\e}| + |\partial_{\e} b_{\e}| \right) (1 + |v|^2) \left\{g(v) + M_{\e}(v) \right\},
\end{align*}
where for the last inequality, we used the inequalities $|x-y| \leq x+y$ and $\varphi_{\e}^{-1}(x) \leq x$ for $x,y \geq 0$. 

\medskip

\noindent On the other hand, remarking that
\begin{equation} \label{eq:derivphieps}
 \forall \, x \geq 0, \qquad \partial_{\e}(\vi_{\e}^{-1})(x) = \partial_{\e} \left(\frac{x}{1 + \e x} \right) = - \left(\frac{x}{1 + \e x} \right)^2 = - \vi_{\e}^{-1}(x)^2,
\end{equation}
we obtain, as $x \mapsto \varphi_{\e}^{-1}(x)^2$ is increasing,
\begin{align*}
\left|\int_{M_{\e}}^{g} \frac{\partial_{\e}(\vi_{\e}^{-1})(g) -\partial_{\e}(\vi_{\e}^{-1})(y)}{y} \, \dd y \right| =  \int_{M_{\e}}^{g} \frac{\vi_{\e}^{-1}(g)^2 -\vi_{\e}^{-1}(y)^2}{y} \, \dd y.
\end{align*}
In the last equation, as well as in the four following ones, we omit writing the dependence of $g$ and $M_{\e}$ in $v$ for clarity. From the identity $|x^2 - y^2| = (x + y) |x-y|$, for $x,y\geq 0$, the fact that $\|\varphi_{\e}^{-1}\|_{\infty} = \e^{-1}$ and that $\varphi_{\e}^{-1}$ is increasing, we get
$$
\int_{M_{\e}}^{g} \frac{\vi_{\e}^{-1}(g)^2 -\vi_{\e}^{-1}(y)^2}{y} \, \dd y  \leq 2 \e^{-1}  \int_{M_{\e}}^{g} \frac{\vi_{\e}^{-1}(g) -\vi_{\e}^{-1}(y)}{y} \, \dd y,
$$
which, using the bound $|\vi_{\e}^{-1}(x) - \vi_{\e}^{-1}(y)| \leq |x-y|$ for $x,y \geq 0$ and $\e > 0$, is less than
$$
2 \e^{-1}  \int_{M_{\e}}^{g} \frac{g -y}{y} \, \dd y.
$$
By direct integration, the latter equals
$$
2 \e^{-1} \left(g \log g - g \log M_{\e} - g + M_{\e}\right),
$$
which, recalling that $M_{\e}(v) = \exp \left(a_{\e} + |v-u_{\e}|^2 \right)$ and $g \geq 0$, is itself less than
$$
2 \e^{-1} \left(g \log g - a_{\e} \, g  + M_{\e}\right).
$$
All in all, as $-a_{\e}g \leq |a_{\e}| g$, we have for any $\e > 0$ and $v \in \R^3$ that
\begin{multline} \label{eqproofprop6}
\left| \frac{\dd}{\dd \e}  K(\e,M_{\e}(v)) \right| \leq \\ 2 \e^{-1} \left\{g \log g(v) + |a_{\e}| g(v)  + M_{\e}(v)\right\} + \left( |\partial_{\e} \alpha_{\e}| + |\partial_{\e} \beta_{\e}| + |\partial_{\e} b_{\e}| \right) (1 + |v|^2) \left\{g(v) + M_{\e}(v) \right\}.
\end{multline}
Let us now fix $\e > 0$ and an arbitrary $\delta \in (0,\e)$. Since the application $\varepsilon \mapsto (\alpha_{\e}, \beta_{\e}, b_{\e})$ is $\mathcal{C}^1$ on $\R_+^*$, there exists $c_0 > 0$ such that
$$
\sup_{\e_* \in (\e - \delta, \e + \delta)}  \left( |\partial_{\e_*} \alpha_{\e_*}| + |\partial_{\e_*} \beta_{\e_*}| + |\partial_{\e_*} b_{\e_*}| \right) \leq c_0.
$$
Moreover, as $\e \mapsto a_{\e}$ is continuous on $\R_+$, there is a $c_1 > 0$ such that $$\sup_{\e_* \in (\e - \delta, \e + \delta)} |a_{\e_*}| \leq c_1.$$ Finally, from Lemma~\ref{lem:bdMeps} in Appendix~\ref{appendix:technicallemmas} (with $\overline{\e} = \e + \delta$ in the lemma), there exist $c_2>0$ and $\eta > 0$ such that for any $v \in \R^3$,
$$
\sup_{\e_* \in [0, \e + \delta]} M_{\e_*}(v) \leq c_2 \, e^{-\eta |v|^2}.
$$
Then~\eqref{eqproofprop6} implies that for any $v \in \R^3$,
\begin{multline*}
   \sup_{\e_* \in (\e-\delta,\e+\delta) }\left|\frac{\dd}{\dd \e_*} \, K(\e_*,M_{\e_*}(v)) \right|  \leq \\ 
   2 (\e-\delta)^{-1} \left(g \log g(v) + c_1 g(v)  + c_2 \, e^{-\eta |v|^2}\right) + c_0 (1 + |v|^2) \left\{g(v) + c_2 \, e^{-\eta |v|^2} \right\}.  
\end{multline*}
Since $g \in L^1_2(\R^3) \cap L \log L(\R^3)$, the right-hand side of the above equation is integrable, and we conclude, by dominated convergence, to the differentiability of $R_g$ and the validity of the differentiation in $\e$ under the integral, ending the proof.
\end{proof}

\subsection{Continuity at zero} The second and last argument required in the proofs of Theorem~\ref{theorem:FD} and Proposition~\ref{prop:FDupper} is the continuity of $\e \mapsto R_g(\e)$ at the point $\e = 0$, which is given in the following lemma.

\begin{lemma} \label{lemma:continuityatzero}
    Let $0 \leq g \in L^1_2(\R^3) \cap L \log L(\R^3)  \setminus \{0\}$. Then the function $\e \in \R_+ \mapsto R_g(\e)$ defined in~\eqref{eqdef:reps} is continuous at the point $\e = 0$.
\end{lemma}

\begin{proof}
We write $R_g(\e)$ as
\begin{equation}
    R_g(\e) = H_{\e} \left( \frac{g}{1 + \e g} \, \right) - H_{\e} \left( \M_{\e}^{\frac{g}{1 + \e g}} \, \right).
\end{equation}
We show that both terms, in the right-hand-side, are continuous at the point $\e = 0$. Firstly, by definition of the Fermi-Dirac entropy, it holds for any $\e > 0$ that
\begin{align*}
H_{\e} \left( \frac{g}{1 + \e g} \, \right) &= \int_{\R^3} \left\{\frac{g}{1 + \e g} \log \left( \frac{g}{1 + \e g} \right) + \e^{-1} \left(1 - \e \frac{g}{1 + \e g}\right) \log \left(1 - \e \frac{g}{1 + \e g}\right) \right\} \dd v \\
&= \int_{\R^3} \left\{\frac{g}{1 + \e g} \log g - \e^{-1} \log(1 + \e g) \right\} \dd v.
\end{align*}
As $g\geq 0$, we have for any $\e >0$ that
\begin{align*}
\left| \frac{g}{1 + \e g} \log g \right| \leq  g |\log g| \qquad \text{and} \qquad \left|-\e^{-1} \log(1 + \e g) \right|  \leq g,
\end{align*}
Since by hypothesis $0 \leq g \in L^1_2(\R^3) \cap L \log L(\R^3)$, the function $g |\log g| + g$ is integrable, and, by dominated convergence, we then conclude to the limit
$$
H_{\e} \left( \frac{g}{1 + \e g} \, \right) \xrightarrow{\e \to 0}  \int_{\R^3} \left\{g \log g - g \right\} \dd v = H_0(g).
$$
Secondly, going through the same computations as for the first term, we have, for any $\e > 0$,
\begin{align*}
 H_{\e} \left( \M_{\e}^{\frac{g}{1 + \e g}}  \right) = \int_{\R^3} \left\{\frac{M_{\e}}{1 + \e M_{\e}} \log M_{\e} - \e^{-1} \log(1 + \e M_{\e} ) \right\} \dd v,
\end{align*}
where we denoted $M_{\e} = \varphi_{\e} \left(\M_{\e}^{\frac{g}{1 + \e g}}  \right)$. Lemma~\ref{lem:bdMeps} in Appendix~\ref{appendix:technicallemmas} (with $\overline{\e} = 1$ in the lemma) provides the existence of $C, \eta > 0$  such that for all $0\leq \e \leq 1$,
$$
|\log M_{\e}(v)| \leq C (1 + |v|^2) \qquad \text{and} \qquad M_{\e} (v) \leq C \, e^{-\eta |v|^2}.
$$
Hence, for any $0 \leq \e \leq 1$ and $v \in \R^3$,
\begin{multline*}
\left|\frac{M_{\e}(v)}{1 + \e M_{\e}(v)} \log M_{\e}(v) \right| \leq   C^2 \, e^{- \eta |v|^2} (1 + |v|^2), \quad \text{and} \qquad \left| \e^{-1} \log(1 + \e M_{\e}(v) ) \right|  \leq  M_{\e}(v)  \leq C \, e^{-\eta |v|^2}.
\end{multline*}
By dominated convergence, it thus comes that
$$
\lim_{\e \to 0} H_{\e} \left( \M_{\e}^{\frac{g}{1 + \e g}} \, \right) = \int_{\R^3} \lim_{\e \to 0} \left( \frac{M_{\e}}{1 + \e M_{\e}} \log M_{\e} -   \e^{-1} \log(1 + \e M_{\e} )\right) \dd v.
$$
As Lemma~\ref{lemma:continuity} ensures that $\displaystyle \lim_{\e \to 0} M_{\e}(v) = \M^g_0 (v)$ for all $v \in \R^3$, the above limit becomes
$$
\int_{\R^3} \M^g_0 \log \M^g_0 - \M^g_0 \, \dd v,
$$
that is $H_0 (\M^g_0)$, ending the proof.
\end{proof}

\medskip

\noindent We can finally conclude to the announced inequalities of Theorem~\ref{theorem:FD} and Proposition~\ref{prop:FDupper}.

\subsection{Final proof of Theorem~\ref{theorem:FD}} Let $\e > 0$ and $0 \leq f \in L^1_2(\R^3) \setminus \{0\}$ such that
$$
g := \frac{f}{1 - \e f}\in  L^1_2(\R^3) \cap L \log L(\R^3) \setminus \{0\}.
$$
Proposition~\ref{prop:rgnonincr} ensures that the application $R_g$ defined by~\eqref{eqdef:reps} is $\mathcal{C}^1$ on $\R_+^*$, and combining~\eqref{eq:Rgprime} with~\eqref{eq:derivphieps} yields, for any $\e_* > 0$,
$$
R'_g(\e_*) = -\int_{\R^3} \int_{M_{\e_*}}^{g} \frac{\vi_{\e_*}^{-1}(g)^2 - \vi_{\e_*}^{-1}(y)^2}{y} \, \dd y \, \dd v.
$$
As, for any $\e_*>0$, the application $x \mapsto \vi_{\e_*}^{-1}(x)^2$ is increasing, the integral in $y$ is 
nonnegative, hence so is the integral in $y$ and $v$, therefore
$$
\forall \, \e_* > 0, \qquad  \qquad R_g'(\e_*) \leq 0. \qquad \qquad  \phantom{\forall \, \e_* > 0,}
$$
Combining this result with Lemma~\ref{lemma:continuityatzero}, which states that $\e_* \mapsto R_g(\e_*)$ is continuous at the point $\e_* = 0$, $R_g$ is then nonincreasing on $\R_+$, which implies in particular that
$$
R_g(0) \geq R_g(\e),
$$
that is
$$
H_{0} \left[g \left|M^{g} \right. \right] \geq H_{\varepsilon} \left[\varphi_{\varepsilon}^{-1}(g) \left|\M_{\varepsilon}^{\varphi_{\varepsilon}^{-1}(g)} \right. \right],
$$
or the announced inequality~\eqref{eq:theorementropiesFD2},
$$
H_{0} \left[\frac{f}{1- \e f} \left|M^{\frac{f}{1- \e f}} \right. \right] \geq H_{\varepsilon} \left[f\left|\M_{\varepsilon}^{f} \right. \right].
$$

\subsection{Final proof of Proposition~\ref{prop:FDupper}} Let us fix $\e > 0$, $\kappa_0 \in (0,1)$ and $0 \leq f \in L^1_2(\R^3) \setminus \{0\}$ such that $1 - \e f \geq \kappa_0$. We let
$$
g := \frac{f}{1 - \e f}.
$$
Then $0 \leq g  \in  L^1_2(\R^3) \cap L^{\infty}(\R^3) \setminus \{0\}$, which is in particular included in $L \log L (\R^3)$. We apply the results of Proposition~\ref{prop:rgnonincr} and Lemma~\ref{lemma:continuityatzero} with the function $g$, ensuring that the application $R_g$, defined by~\eqref{eqdef:reps}, is continuous on $\R_+$ and $\mathcal{C}^1$ on $\R_+^*$ with, combining~\eqref{eq:Rgprime} with~\eqref{eq:derivphieps},
$$
\forall \, \e_* > 0, \qquad R'_g(\e_*) = -\int_{\R^3} \int_{M_{\e_*}}^{g} \frac{\vi_{\e_*}^{-1}(g)^2 - \vi_{\e_*}^{-1}(y)^2}{y} \, \dd y \, \dd v.
$$
For $\e_*>0$, $v \in \R^3$ and any $y$ between $M_{\e_*}(v)$ and $g(v)$, one has (omitting the dependencies in $v$ for clarity)

\smallskip

$$
|\vi_{\e_*}^{-1}(g)^2 - \vi_{\e_*}^{-1}(y)^2| =  (\vi_{\e_*}^{-1}(g)+ \vi_{\e_*}^{-1}(y)) |\vi_{\e_*}^{-1}(g) - \vi_{\e_*}^{-1}(y)| \leq c_{g}(\varepsilon_*) \,  |\vi_{\e_*}^{-1}(g) - \vi_{\e_*}^{-1}(y)|,
$$
where we denoted
\begin{equation*}
c_{g}(\varepsilon_*) := \|\vi_{\e_*}^{-1}(g)\|_{\infty} + \max \left(\|\vi_{\e_*}^{-1}(g)\|_{\infty}, \|\M_{\e_*}^{\vi_{\e_*}^{-1}(g)}\|_{\infty} \right).
\end{equation*}

\smallskip

\noindent As both $z \mapsto \vi_{\e_*}^{-1}(z)^2$ and $z \mapsto \vi_{\e_*}^{-1}(z)$ are increasing, we thus have
\begin{equation*} 
     R'_g(\e_*) \geq - c_{g}(\varepsilon_*) \int_{\R^3} \int_{M_{\e_*}}^{g} \frac{\vi_{\e_*}^{-1}(g) - \vi_{\e_*}^{-1}(y)}{y} \, \dd y \, \dd v =  - c_{g}(\varepsilon_*) \, R_g(\e_*).
\end{equation*}
Noticing that
\begin{equation*} 
c_g(\e_*) \leq \|g\|_{\infty} + \max \left(\|g\|_{\infty} + \|\M_{\e_*}^{\varphi_{\e_*}^{-1}(g)} \|_{\infty} \right),
\end{equation*}
we conclude to the relationship, for any $\e_* > 0$,
\begin{equation}\label{eq:relatRprimeR}
 R'_g(\e_*) \geq - \left( \|g\|_{\infty} + \max \left(\|g\|_{\infty} + \|\M_{\e_*}^{\varphi_{\e_*}^{-1}(g)} \|_{\infty} \right) \right) \, R_g(\e_*).
\end{equation}
By hypothesis, it holds that $1- \e f \geq \kappa_0$, hence \begin{equation}  \label{eq:ginftykappa0} \|g\|_{\infty} \leq \e^{-1} \, (\kappa_0^{-1} - 1), \end{equation} and we have for any $\e_* > 0 $ that
\begin{equation} \label{eqproofProp2_1}
\varepsilon_* \varphi_{\e_*}^{-1}(g) = \frac{\e_* g}{1 + \e_* g} \leq \left(1 + \frac{1}{\e_* \|g\|_{\infty}} \right)^{-1} \leq \left(1 + \frac{\e }{\e_*} \, (\kappa_0^{-1} - 1)^{-1} \right)^{-1}.
\end{equation}
Let us use the notation
$$
\gamma_{\e_*} = \frac{T_{*}}{T_F(\rho_{*}, \e_*)},
$$
where $\rho_{*}$ and $T_*$ are respectively the density and temperature associated to the distribution $\varphi_{\e_*}^{-1}(g)$ defined as in~\eqref{eq:normalizationdef}, and $\displaystyle T_F(\rho_{*}, \e_*) = \frac12 \left( \frac{3 \rho_{*} \e_*}{4 \pi}\right)^{2/3}$ is the corresponding $\e_*$-Fermi temperature. 

\noindent Using Equation~\eqref{eq:propkappa11} in Proposition~\ref{prop:kappa1} in Appendix~\ref{appendix:technicallemmas} and the fact that $\M_{\e_*}^{\varphi_{\e_*}^{-1}(g)} \leq \varphi_{\e_*}\left( \M_{\e_*}^{\varphi_{\e_*}^{-1}(g)}\right)$, we have, whenever $\gamma_{\e_*} \geq \gamma^{\dag} = \left(\frac{4}{\pi} \right)^{\frac13} \left(\frac{5}{3} \right)^{\frac53}$,
\begin{equation}
\e_* \| \M_{\e_*}^{\varphi_{\e_*}^{-1}(g)}  \|_{\infty} \leq \frac23 \left(\frac{\gamma_{\e_*}}{\gamma^{\dag}} \right)^{-3/2}.\label{eqproofProp2_2}
\end{equation}
We now notice that~\cite[Proposition 4]{lu2001spatially} actually implies
$$
\gamma_{\e_*} \geq \frac25 \left(\e_* \|\vi_{\e_*}^{-1}(g)\|_{\infty} \right)^{-2/3},
$$
which, combined with~\eqref{eqproofProp2_1}, yields
\begin{equation} \label{eq:gammaestar}
\gamma_{\e_*} \geq \frac25 \left(1 + \frac{\e }{\e_*} \, (\kappa_0^{-1} - 1)^{-1} \right)^{2/3}.
\end{equation}

\medskip

\noindent We define
\begin{equation} \label{eqdefepsdag}
\e^{\dag} := \left(\left(\frac52 \, \gamma^{\dag} \right)^{\frac32} - 1 \right)^{-1} \, (\kappa_0^{-1} - 1)^{-1} \, \e.
\end{equation}
Then, for any $\e_* \in (0, \e^{\dag} \,]$, combining~\eqref{eq:gammaestar} with~\eqref{eqdefepsdag} yields
$$
\gamma_{\e_*} \geq  \frac25 \left(1 + \frac{\e }{\e^{\dag}} \, (\kappa_0^{-1} - 1)^{-1} \right)^{2/3} =\gamma^{\dag}.
$$
Applying~\eqref{eqproofProp2_2}, along with~\eqref{eq:gammaestar} and the fact that
$$
\frac23 \left(\frac52 \gamma^{\dag} \right)^{\frac32} = \frac23 \left(\frac52\right)^{\frac32}\left(\frac{4}{\pi} \right)^{\frac12} \left(\frac{5}{3} \right)^{\frac52} \leq 11,
$$
gives, for any $\e_* \in (0, \e^{\dag} \,]$,
$$
\|\M_{\e_*}^{\varphi_{\e_*}^{-1}(g)} \|_{\infty} \leq \frac{11}{\e_* + \e  (\kappa_0^{-1} - 1)^{-1}} \leq 11 \, \e^{-1} \, (\kappa_0^{-1}-1).
$$
Therefore, recalling~\eqref{eq:ginftykappa0}, we have for any $\e_* \in (0, \, \e^{\dag} \, ]$ that

$$
\|g\|_{\infty} + \max \left(\|g\|_{\infty},  \|\M_{\e_*}^{\varphi_{\e_*}^{-1}(g)} \|_{\infty}\right) \leq \e^{-1} \, (\kappa_0^{-1}-1) + 11 \e^{-1} \, (\kappa_0^{-1}-1) =  12 \, \e^{-1} \, (\kappa_0^{-1}-1),
$$
implying, with the relationship~\eqref{eq:relatRprimeR}, that for any $\e_* \in (0, \, \e^{\dag}\,]$,

\begin{equation} \label{eq:intermediateproofupperbound}
R_g'(\e_*) \geq -12 \, \e^{-1} \, (\kappa_0^{-1}-1) \, R_g(\e_*).
\end{equation}
Applying a Grönwall argument over~\eqref{eq:intermediateproofupperbound} and using the continuity of $\e_* \mapsto R_g(\e_*)$ at the point $\e_* = 0$ provided by Lemma~\ref{lemma:continuityatzero}, we obtain
\begin{equation} \label{eqproofProp2_4}
\forall \, \e_* \in (0,\, \e^{\dag} \,], \hspace{50pt} R_g(0) \leq \, \exp \left(12 \, \frac{\e_*}{\e} \, (\kappa_0^{-1}-1) \right) R_g(\e_*).\hspace{120pt}
\end{equation}
On the one hand, if $\e \leq \e^{\dag}$, we take $\e_* = \e$ in~\eqref{eqproofProp2_4} to get

\begin{equation} \label{eq:caseinf}
R_g(0) \leq \, \exp \left(12 \, (\kappa_0^{-1}-1) \right)  R_g(\e).
\end{equation}

\medskip

\noindent On the other hand, if $\e \geq \e^{\dag}$, we apply~\eqref{eqproofProp2_4} with $\e_* = \e^{\dag}$ to get

\begin{equation} \label{eqproofProp2_5}
R_g(0) \leq \, \exp \left(12 \, \frac{\e^{\dag}}{\e} (\kappa_0^{-1}-1) \right) R_g(\e^{\dag}).
\end{equation}

\smallskip

\noindent Since it holds for any $\e_*>0$ that $\|\M_{\e_*}^{\varphi_{\e_*}^{-1}(g)} \|_{\infty} \leq \e_*^{-1}$, we have in particular that, for any $\e_* \in [\e^{\dag}, \e]$,

\begin{equation*}
\|\M_{\e_*}^{\varphi_{\e_*}^{-1}(g)} \|_{\infty} \leq (\e^{\dag})^{-1},
\end{equation*}
which, combined with~\eqref{eq:relatRprimeR} and~\eqref{eq:ginftykappa0}, leads to

\begin{equation} \label{eq:intermediateproofupperbound2}
\forall \, \e_* \in [\e^{\dag}, \e], \hspace{30pt} R_g'(\e_*) \geq - \e^{-1} \, \left(  (\kappa_0^{-1}-1) + \max \left(\kappa_0^{-1}-1, \, \frac{\e}{\e^{\dag}} \right)\right) R_g(\e_*). \hspace{100pt}
\end{equation}
By a Grönwall argument over~\eqref{eq:intermediateproofupperbound2}, we obtain

\begin{equation} \label{eq:RprimRedag}
R_g(\e^{\dag}) \leq \exp \left(\frac{\e-\e^{\dag}}{\e} \, \left(\kappa_0^{-1}-1 + \max \left(\kappa_0^{-1}-1, \, \frac{\e}{\e^{\dag}} \right) \right) \right) R_g(\e).
\end{equation}
Recalling that $\e^{\dag}$ is defined by~\eqref{eqdefepsdag}, we have
$$
\frac{\e}{\e^{\dag}} = \left(\left(\frac52 \, \gamma^{\dag} \right)^{\frac32} - 1 \right) \, (\kappa_0^{-1} - 1).
$$
As $\gamma^{\dag} = \left(\frac{4}{\pi} \right)^{\frac13} \left(\frac{5}{3} \right)^{\frac53}$, we have
$\left(\frac52 \, \gamma^{\dag} \right)^{\frac32} - 1  \leq 15,$ therefore

\smallskip

\begin{equation} \label{eq:bdmaxproof1}
    \left(\kappa_0^{-1}-1 + \max \left(\kappa_0^{-1}-1, \, \frac{\e}{\e^{\dag}} \right) \right) \leq 16 \, (\kappa_0^{-1} - 1).
\end{equation}
Combining~\eqref{eq:RprimRedag} with~\eqref{eq:bdmaxproof1}, we get

\smallskip

\begin{equation} \label{eq:RprimRedag2}
R_g(\e^{\dag}) \leq \exp \left(\frac{\e-\e^{\dag}}{\e} \, 16\left(\kappa_0^{-1}-1\right)  \right) R_g(\e).
\end{equation}
Finally, combining~\eqref{eqproofProp2_5} with~\eqref{eq:RprimRedag2}, we obtain, in this case where we assumed $\e \geq \e^{\dag}$,
$$
R_g(0) \leq \exp \left( \left\{12\frac{\e^{\dag}}{\e} + 16\frac{\e - \e^{\dag}}{\e} \right\}(\kappa_0^{-1}-1) \right) R_g(\e) \leq  \exp \left( 16 \, (\kappa_0^{-1}-1) \right) R_g(\e).
$$
We thus conclude, with~\eqref{eq:caseinf}, that in any case, we have

$$
R_g(0) \leq  \exp \left( 16 \, (\kappa_0^{-1}-1) \right) R_g(\e),
$$

\bigskip

\noindent yielding the required result as $R_g(0) = H_{0} \left[\frac{f}{1- \e f} \left|M^{\frac{f}{1- \e f}} \right. \right]$ and $R_g(\e) = H_{\varepsilon} \left[f\left|\M_{\varepsilon}^{f} \right. \right]$.

\medskip

\section{Proof of Corollary~\ref{theorem:BFDcerci}} \label{section:proofcorollarybfd}

\noindent This section is devoted to the proof of Corollary~\ref{theorem:BFDcerci}. Throughout this section, we use the notation
\begin{equation} \label{qtieshatbis}
\begin{pmatrix}  \rho_{\varepsilon}\\ \rho_{\varepsilon} u_{\varepsilon} \\ 3 \rho_{\varepsilon} T_{\varepsilon} + \rho_{\varepsilon}|u_{\varepsilon}|^2 \end{pmatrix} = \int_{\R^3} \varphi_{\e} (f)(v) \, \begin{pmatrix}  1 \\ v \\ |v|^2 \end{pmatrix} \, \dd v, \qquad \begin{pmatrix}  \rho \\ \rho  u  \\ 3 \rho  T  + \rho |u |^2 \end{pmatrix} = \int_{\R^3} f(v) \, \begin{pmatrix}  1 \\ v \\ |v|^2 \end{pmatrix} \, \dd v.
\end{equation}
In the following, we use the bound
\begin{equation} \label{ineq:Tkappa0}
1 - \varepsilon f \geq \kappa_0 \implies T_{\varepsilon} \geq \kappa_0 T.
\end{equation}
This comes by first remarking that, if $1 - \e f \geq \kappa_0$,
$$
3 \rho_{\varepsilon} T_{\varepsilon} = \int_{\R^3} \varphi_{\varepsilon}(f)(v) |v-u_{\varepsilon}|^2 \, \dd v \geq \int_{\R^3} f(v) |v-u_{\varepsilon}|^2 \, \dd v = \int_{\R^3} f(v) |v-u-(u_{\varepsilon}-u)|^2 \, \dd v = 3 \rho T + \rho |u_{\varepsilon}-u|^2,
$$
implying $\rho_{\e} T_{\e} \geq \rho T$, and the fact that
\begin{equation}
    \label{ineq:rhokappa0}
    \rho_{\varepsilon} = \int_{\R^3} \varphi_{\varepsilon}(f)(v) \, \dd v \leq \kappa_0^{-1} \int_{\R^3} f(v) \, \dd v = \kappa_0^{-1} \, \rho.
\end{equation}

\subsection{Super quadratic kernels}

\begin{proof}
Let $\kappa_0 \in (0,1)$, $\varepsilon > 0$ and $0 \leq f \in L^1_2(\R^3) \setminus \{0\}$ such that $1 - \varepsilon f \geq \kappa_0$. We have from \eqref{eq:compareproductions} that
\begin{equation} \label{eq1proofsq}
\Deps (f) \geq  \kappa^4_0 \, \D_0(\varphi_{\varepsilon}(f)).
\end{equation}  
We apply~\cite[Theorem 2.1]{villani2003} to the function $\varphi_{\varepsilon}(f)$, adapted (since it is written for normalized functions) with the change of variable $w = \sqrt{T_{\varepsilon}} \; v + u_{\e}$, and obtain
\begin{equation*}
\D_0(\varphi_{\varepsilon}(f)) \geq \min(1,T_{\varepsilon}) \frac{|\Sb^2|}{28} \,  (3 \rho_{\varepsilon} T_{\varepsilon} -  \rho_{\varepsilon} T^*(\varphi_{\varepsilon}(f))) \; H_0 \left[\varphi_{\varepsilon}(f) \Big| M^{\varphi_{\varepsilon}(f)} \right],
\end{equation*}
where $ T^*(g)$ stands for the maximal directional temperature of $g$, defined by $\frac{\underset{e \in \Sb^2}{\sup} \int_{\R^3} g(v) \, (v \cdot e)^2 \, \dd v}{\int_{\R^3} g(v) \, \dd v}$. Since in particular,
$$
3 \rho_{\varepsilon} T_{\varepsilon} -  \rho_{\varepsilon} T^*(\varphi_{\varepsilon}(f)) \geq 2 \rho_{\varepsilon} T_*(\varphi_{\varepsilon}(f)),
$$
where $T_*(g)$ stands for the minimal directional temperature of $g$, defined in \eqref{eqdef:minTi}, we obtain
\begin{equation} \label{eq2proofsq}
\D_0(\varphi_{\varepsilon}(f)) \geq   \min(1,T_{\varepsilon})\, \frac{2\pi }{7} \, \rho_{\varepsilon} T_{*}(\varphi_{\varepsilon}(f)) \; H_0 \left[\varphi_{\varepsilon}(f) \Big| M^{\varphi_{\varepsilon}(f)} \right].
\end{equation} 
Applying Theorem~\ref{theorem:FD}'s lower-bound inequality \eqref{eq:theorementropiesFD2}, we obtain
\begin{equation} \label{eq3proofsq}
H_0 \left[\varphi_{\varepsilon}(f) \Big| M^{\varphi_{\varepsilon}(f)} \right] \geq  H_{\varepsilon}[f|\Me].
\end{equation}
Combining \eqref{eq1proofsq}--\eqref{eq3proofsq} yields
$$
\Deps (f) \geq  \frac{2\pi}{7} \, \kappa_0^4 \,  \min(1, T_{\varepsilon}) \, \rho_{\varepsilon} \, T_{*}(\varphi_{\varepsilon}(f)) \; H_{\varepsilon}[f|\Me].
$$
The required result is obtained after remarking that, from \eqref{ineq:Tkappa0},  $T_{\varepsilon} \geq \kappa_0 T$, and 
$$
\rho_{\varepsilon} \,T_{*}(\varphi_{\varepsilon}(f)) = \underset{e \in \Sb^2}{\inf} \int_{\R^3} \varphi_{\varepsilon}(f)(v) \, (v \cdot e)^2 \, \dd v  \geq \underset{e \in \Sb^2}{\inf} \int_{\R^3} f(v) \, (v \cdot e)^2 \, \dd v  =  \rho \, T_{*}(f).
$$
\end{proof}

\subsection{General kernels}

\begin{proof}
Let $\kappa_0 \in (0,1)$, $\varepsilon > 0$ and $0 \leq f \in L^1_2(\R^3) \setminus \{0\}$ such that $1 - \varepsilon f \geq \kappa_0$. Again, \eqref{eq:compareproductions} yields
$$
\Deps (f) \geq  \kappa^4_0 \, \D_0(\varphi_{\varepsilon}(f)).
$$
Remarking that $\varphi_{\varepsilon}(f)(v) \geq f(v) \geq K_0 \, e^{-A_0 |v|^{q_0}}$, we apply~\cite[Theorem 4.1]{villani2003} to the function  $\varphi_{\varepsilon}(f)$ and obtain for any $\alpha \in (0,1)$ the existence of some $\overline{K}_{\alpha}\left(\varphi_{\varepsilon}(f) \right)$ depending only on $\alpha$, upper and lower bounds on $\rho_{\varepsilon}$, $T_{\varepsilon}$ (see Remark~\ref{remark:proofgeneralkernel} below), $q_0$ and upper bounds on $A_0$, $1/K_0$, $\|\varphi_{\varepsilon}(f)\|_{L^1_s}$ and $\|\varphi_{\varepsilon}(f)\|_{H^k}$ where $s = s(\alpha,q_0,\beta_+,\beta_-)$ and $k=k(\alpha,s,\beta_+,\beta_-)$ such that
$$
\D_0(\varphi_{\varepsilon}(f)) \geq \overline{K}_{\alpha}\left(\varphi_{\varepsilon}(f) \right) H_0 \left[\varphi_{\varepsilon}(f) \Big| M^{\varphi_{\varepsilon}(f)} \right]^{1+ \alpha}.
$$
Again, Theorem~\ref{theorem:FD}'s lower-bound inequality \eqref{eq:theorementropiesFD2} implies
$$
H_0 \left[\varphi_{\varepsilon}(f) \Big| M^{\varphi_{\varepsilon}(f)} \right] \geq  H_{\varepsilon}[f|\Me],
$$
so that
$$
\Deps (f) \geq \kappa_0^4 \, \overline{K}_{\alpha}\left(\varphi_{\varepsilon}(f) \right) \, \, H_{\varepsilon}[f|\Me]^{1 + \alpha}.
$$
Moreover,  we can upper-bound $\|\varphi_{\varepsilon}(f)\|_{L^1_s}$ by $\kappa_0^{-1} \|f\|_{L^1_s}$, and $\|\varphi_{\varepsilon}(f)\|_{H^k}$ by a polynomial in $(\|f\|_{H^l})_{l \leq k}$, $\kappa_0^{-1}$ and an upper bound on $\varepsilon$, since for $l \geq 1$, and $x$ such that $1 - \varepsilon x \geq \kappa_0$,
$$
|\varphi_{\varepsilon}^{(l)}(x)| = \frac{ \, l! \, \varepsilon^{l-1}}{(1-\varepsilon x)^{l+1}} \leq l! \, \kappa_0^{-l-1} \, \varepsilon^{l-1}.
$$
 
\noindent We end the proof by using
\begin{itemize}
    \item Equations~\eqref{ineq:Tkappa0}--\eqref{ineq:rhokappa0} stating that $\rho_{\varepsilon} \leq \kappa_0^{-1} \, \rho$ and $T_{\varepsilon} \geq \kappa_0 \, T$,
    \item the fact that $\rho \leq \rho_{\varepsilon}$, coming from
    $$
    \int_{\R^3} f \, \dd v \leq \int_{\R^3} \frac{f}{1- \e f} \, \dd v, 
    $$
        \item the fact that $T_{\varepsilon} \leq \kappa_0^{-1} \, T$, coming from
$$
3 \rho_{\varepsilon} T_{\varepsilon} + \rho_{\varepsilon} |u_{\varepsilon}-u|^2 = \int_{\R^3} \varphi_{\e}(f)(v) \, |v-u|^2\, \dd v \leq \kappa_0^{-1} \int_{\R^3} f(v) \, |v-u|^2\, \dd v =  \kappa_0^{-1} \, 3 \rho T.
$$
implying $\rho_{\e} T_{\e} \leq \kappa_0^{-1} \rho T$ hence $T_{\e} \leq \kappa_0^{-1} \frac{\rho}{\rho_{\e}} T \leq \kappa_0^{-1} T$.
\end{itemize}
\end{proof}  
\begin{remark} \label{remark:proofgeneralkernel}
In~\cite[Theorem 4.1]{villani2003}, the constant $\Bar{K}_{\alpha} \left(\varphi_{\varepsilon}(f) \right)$ is said to depend on $\rho_{\varepsilon}$, $T_{\varepsilon}$. Having a look at the proof there, this dependence can be relaxed into a dependence on upper and lower bounds of those quantities.
\end{remark}

\subsection{Inverse result}
Let $\varepsilon \in (0, \frac12)$, $B_0 > 0$ and $0 \leq \beta < 2$. We assume that
$$
\int_{\Sb^2} B(v,v_*,\sigma) \, \dd \sigma \leq B_0 \, (1 + |v-v_*|^{\beta}).
$$
Relying on Bobylev and Cercignani's~\cite[Theorem 1]{bobylev1999}, and the nice formulation of this theorem by Villani~\cite[Theorem 1.1]{villani2003}, we know that there exists a family of normalized functions $(g_l)$, in the sense that for any~$l$,
\begin{equation} \label{eq:glisnormalized}
\int_{\R^3} g_l(v) \begin{pmatrix}
1 \\ v \\ |v|^2
\end{pmatrix} \dd v = \begin{pmatrix}
1 \\ 0 \\ 3
\end{pmatrix},
\end{equation}
such that (taking $\delta = 1 - \frac{1+2 \varepsilon}{2} > 0$ in \cite[Theorem 1.1]{villani2003}) 
\begin{equation} \label{eq:bdsglpOi}
\forall v \in \R^3, \qquad 2 \geq  g_l(v) \geq \frac{1+2 \varepsilon}{2} \, (2 \pi)^{-3/2}\, e^{-|v|^2},
\end{equation}
and
\begin{equation}
    \frac{\D_0(g_l)}{H_0[g_l|M^{g_l}]} \underset{l \to \infty}{\longrightarrow} 0.
\end{equation}
We now consider
$$
f_l^{\varepsilon} = \frac{g_l}{1 + \varepsilon g_l} \equiv \varphi_{\varepsilon}^{-1}(g_l).
$$
As $\varphi_{\e}^{-1}$ is increasing, and from~\eqref{eq:bdsglpOi}, we have
$$
f_l^{\varepsilon}(v) \equiv \frac{g_l(v)}{1 + \varepsilon g_l(v)} \geq \frac{g_l(v)}{1 + 2 \varepsilon} \geq \frac{1+2 \varepsilon}{2 (1 + 2 \varepsilon)} \, (2 \pi)^{-3/2}\, e^{-|v|^2} = \frac{1}{2} \, (2 \pi)^{-3/2}\, e^{-|v|^2}.
$$
We denote
$$
\int_{\R^3} f_l^{\varepsilon}(v) \begin{pmatrix}
1 \\ v \\ |v|^2
\end{pmatrix} \dd v = \begin{pmatrix}
\rho^l_{\varepsilon} \\ \rho^l_{\varepsilon} u^l_{\varepsilon} \\ 3 \rho^l_{\varepsilon} T^l_{\varepsilon} + \rho^l_{\varepsilon} |u^l_{\varepsilon}|^2
\end{pmatrix},
$$
We have, from $f_l^{\varepsilon} = \frac{g_l}{1 + \varepsilon g_l}$ and $g_l \leq 2$,
$$
\frac{1}{1 + 2 \varepsilon}\int_{\R^3} g_l(v) \, \dd v  \leq \int_{\R^3} f_l^{\varepsilon}(v) \, \dd v \leq \int_{\R^3} g_l(v) \, \dd v,
$$
that is
$$
\frac{1}{1 + 2 \varepsilon} \leq \rho^l_{\varepsilon} \leq 1.
$$
Secondly,
$$
3 \rho_{\varepsilon}^l T^l_{\varepsilon} \leq 3 \rho^l_{\varepsilon} T^l_{\varepsilon} + \rho^l_{\varepsilon} |u^l_{\varepsilon}|^2 = \int_{\R^3} f_l^{\varepsilon}(v) \,  |v|^2 \, \dd v \leq \int_{\R^3} g_l(v) \,  |v|^2 \, \dd v  = 3,
$$
that is $\rho_{\varepsilon}^l T^l_{\varepsilon} \leq 1$, hence
$$
 T^l_{\varepsilon} \leq \frac{1}{\rho_{\varepsilon}^l} \leq 1 + 2 \varepsilon.
$$
On the other hand, since $g_l$ is normalized by~\eqref{eq:glisnormalized}, we have
$$
3 \rho_{\varepsilon}^l T^l_{\varepsilon} = \int_{\R^3} f_l^{\varepsilon}(v) \,  |v-u_{\varepsilon}^l|^2 \, \dd v  \geq \frac{1}{1 + 2 \varepsilon} \int_{\R^3} g_l(v) \,  |v-u_{\varepsilon}^l|^2 \, \dd v   = \frac{3 + |u_{\varepsilon}^l|^2}{1 + 2 \varepsilon} \geq \frac{3}{1 + 2 \varepsilon},
$$
that is $\rho_{\varepsilon}^l T^l_{\varepsilon} \geq \frac{1}{1 + 2 \varepsilon}$, hence
$$
 T^l_{\varepsilon} \geq \frac{1}{\rho_{\varepsilon}^l} \times \frac{1}{1 + 2 \varepsilon} \geq \frac{1}{1 + 2 \varepsilon}.
$$
Finally, since $f_l^{\varepsilon} = g_l - \varepsilon \, \frac{g_l^2}{1 + \varepsilon g_l}$, $g_l$ is normalized by~\eqref{eq:glisnormalized} and $g_l \leq 2$, we have
\begin{align*}
\rho^l_{\varepsilon} \, |u^l_{\varepsilon}| &= \left| \int_{\R^3} f_l^{\varepsilon}(v) \, v \, \dd v \right| = \varepsilon \left| \int_{\R^3} \frac{g_l(v)^2}{1 + \varepsilon g_l(v)} \, v \, \dd v \right| \leq \varepsilon \int_{\R^3} \frac{g_l(v)^2}{1 + \varepsilon g_l(v)} \, |v| \, \dd v \leq 2 \varepsilon \int_{\R^3} g_l(v) \, |v| \, \dd v \\
&\leq 2 \varepsilon \left(\int_{\R^3} g_l(v) \, \dd v \right)^{1/2} \left(\int_{\R^3} g_l(v) \, |v|^2 \, \dd v \right)^{1/2} = 2 \sqrt{3} \, \varepsilon,
\end{align*}
where the last inequality is a Cauchy-Schwarz argument. We then conclude, since $\varepsilon \in (0,\frac12)$, that
$$
|u^l_{\varepsilon}| \leq \frac{1}{\rho^l_{\varepsilon}} \, 2 \sqrt{3} \, \varepsilon \leq 2 \sqrt{3} \, (1 + 2 \varepsilon) \, \varepsilon \leq 4 \sqrt{3} \, \varepsilon.
$$
All in all, we proved that for any $l$, $f_l^{\varepsilon} \in \mathcal{C}^{\varepsilon}_{1,0,1}$.

\medskip

\noindent We now apply Equation \eqref{eq:compareproductions} to get
$$
\Deps (f_l^{\varepsilon}) \leq \D_0(g_l),
$$
and Proposition~\ref{prop:FDupper}'s upper-bound inequality, Equation \eqref{eq:theorementropiesFD}, which gives, for any $\varepsilon \in (0,\frac12)$, and denoting $\kappa^{\varepsilon}_0 = 1 - 2 \varepsilon$ (so that $1 - \varepsilon f^{\varepsilon}_l \geq \kappa_0^{\varepsilon}$),
$$
H_{\varepsilon} \left[f_l^{\varepsilon} \left|\M_{\varepsilon}^{f_l^{\varepsilon}} \right.\right] \geq \frac{1}{C_0(\kappa^{\varepsilon}_0)} \, H_0 \left[ \frac{f_l^{\varepsilon}}{1 - \varepsilon f_l^{\varepsilon}} \left|M^{\frac{f_l^{\varepsilon}}{1 - \varepsilon f_l^{\varepsilon}}} \right. \right] = \frac{1}{C_0(\kappa^{\varepsilon}_0)} \, H_0[g_l|M^{g_l}].
$$
We finally obtain that, for any $\varepsilon \in (0 ,\frac12)$,
\begin{equation}
    \frac{\Deps \left(f_l^{\varepsilon} \right)}{H_{\varepsilon} \left[f_l^{\varepsilon} \left|\M_{\varepsilon}^{f_l^{\varepsilon}} \right.\right]} \leq \frac{1}{C_0(\kappa^{\varepsilon}_0)} \times \frac{\D_0(g_l)}{H_0[g_l|M^{g_l}]} \underset{l \to \infty}{\longrightarrow} 0.
\end{equation}

\section{About the Landau-Fermi-Dirac equation} \label{section:LFD}
\noindent A study of the entropy-entropy production relationship for the Landau-Fermi-Dirac equation for various cross sections was recently conducted by Alonso, Bagland, Desvillettes and Lods in~\cite{ABDLentropy,ABDLsoft} (see also Desvillettes~\cite{desvillettesLandauhardpotential} and Alonso, Bagland, Lods~\cite{ABL}) and we do not (intend to) provide here any new result on this side. We refer the interested reader to these papers and the references therein for a detailed study of the Landau-Fermi-Dirac equation. Nevertheless, using the exact same strategy of transposing results from the Classical towards the Fermi-Dirac case that we used in the Boltzmann-Fermi-Dirac case (Corollary~\ref{theorem:BFDcerci}), we obtain an original proof of inequalities reminiscent of the ones stated in~\cite{ABDLentropy,ABDLsoft} and~\cite{desvillettesLandauhardpotential}, namely~\cite[Theorem 1.4]{ABDLentropy}, the first inequality in the proof of~\cite[Proposition 5.8]{ABDLsoft} and~\cite[Proposition~2]{desvillettesLandauhardpotential}. Let us first briefly introduce the Landau-Fermi-Dirac equation.

\medskip

\noindent In dimension $3$, the Landau-Fermi-Dirac operator writes, for $\varepsilon \geq 0$ and nonnegative $f \in L^1_2(\R^3)$ such that $1 - \varepsilon f \geq 0$,
\begin{equation} \label{eq:LFDoperator}
    Q^L_{\varepsilon}(f)(v) := \nabla_v \cdot \int_{\R^3} \Psi(|v-v_*|) \Pi(v-v_*) \Big[f_* (1 - \varepsilon f_*) \nabla f -f (1 - \varepsilon f) \nabla f_* \Big]\, \dd v_*.
\end{equation}
The case $\varepsilon = 0$ corresponds to the classical Landau operator. In \eqref{eq:LFDoperator}, we used the common short-hands $f \equiv f(v)$ and $f_* \equiv f(v_*)$, $\Pi(z)$ denotes the orthogonal projection on $(\R z)^{\perp}$, whose components are
$$
\Pi_{ij}(z) = \delta_{ij} - \frac{z_i z_j}{|z|^2},
$$
and $\Psi$ is sometimes called the kinetic potential. The entropy production associated to the Landau-Fermi-Dirac operator writes
\begin{equation} \label{eq:entropyproductionLFD}
    \Deps^L(f) = \frac12 \int_{\R^3 \times \R^3} \Psi(|v-v_*|) \, f f_* (1- \varepsilon f) (1 - \varepsilon f_*) \left| \Pi(v-v_*) \left(\frac{\nabla f}{f(1-\varepsilon f)} - \frac{\nabla f_*}{f_*(1-\varepsilon f_*)} \right) \right|^2 \, \dd v \, \dd v_*.
\end{equation}
We notice that \eqref{eq:entropyproductionLFD} can be rewritten (where $\varphi_{\e}$ is defined in~\eqref{eqdef:varphieps})
$$
\Deps^L(f) = \frac12 \int_{\R^3 \times \R^3} \Psi(|v-v_*|) \, \varphi_{\varepsilon}(f) \varphi_{\varepsilon}(f)_* (1- \varepsilon f)^2 (1 - \varepsilon f_*)^2 \left| \Pi(v-v_*) \left(\frac{\nabla \varphi_{\varepsilon}(f)}{\varphi_{\varepsilon}(f)} - \frac{\nabla \varphi_{\varepsilon}(f)_*}{\varphi_{\varepsilon}(f)_*} \right) \right|^2 \, \dd v \, \dd v_*.
$$
This directly implies that for any $\kappa_0 \in (0,1)$, $\varepsilon > 0$ and nonnegative $f \in L^1_2(\R^3) \setminus \{0\}$ such that $1 - \varepsilon f \geq \kappa_0$,

\begin{equation}
   \D^L_0(\varphi_{\varepsilon}(f)) \geq \Deps^L(f) \geq \kappa_0^4 \, \D^L_0(\varphi_{\varepsilon}(f)).
\end{equation}
This inequality, similar to \eqref{eq:compareproductions}, allows to apply the same strategy as in the Boltzmann-Fermi-Dirac case. Adapting the entropy-entropy production inequalities known for the Landau equation \cite[Chapter 3, Theorem 14]{villanireview} and~\cite[Remark 1]{desvillettesLandauhardpotential}, we get the following proposition.

 \begin{proposition}  \label{theorem:LFDcerci} 
 
\textbf{Entropy-entropy production inequalities for the Landau-Fermi-Dirac operator} \emph{(adaptation of~\cite[Chapter 3, Theorem 14]{villanireview} and~\cite[Remark 1]{desvillettesLandauhardpotential})}.
We recall that $\D_{\varepsilon}^L$, $H_{\varepsilon}$ and $T_*(f)$ are defined respectively in~\eqref{eq:entropyproductionLFD}, \eqref{eqdef:entropy}--\eqref{eqdef:entropyrelat} and \eqref{eqdef:minTi}.
 
\medskip

 $\bullet$ \textbf{Over-Maxwellian case \emph{[}Adaptation of~\cite[Chapter 3, Theorem 14, (i)]{villanireview}\emph{]}.} Assume $\Psi(|z|) \geq |z|^2$. Then, for any $\kappa_0 \in (0,1)$, $\varepsilon \geq 0$ and $0 \leq f \in L^1_2(\R^3) \setminus \{0\}$ such that $1 - \varepsilon f \geq \kappa_0$, we have
\begin{equation} \label{eq:cercignaniovermaxlandau}
    \Deps^L (f)  \geq 4 \, \kappa_0^4 \, \rho \, T_{*}(f) \, H_{\varepsilon}[f|\Me],
\end{equation}

\medskip

 $\bullet$ \textbf{Soft potentials \emph{[}Adaptation of~\cite[Chapter 3, Theorem 14, (ii)]{villanireview}\emph{]}.} Assume $\Psi(|z|) \geq |z|^2 (1 + |z|)^{-\beta}$ with $\beta > 0$. Then, for any $\kappa_0 \in (0,1)$, $\varepsilon \geq 0$ and $0 \leq f \in L^1_2(\R^3) \setminus \{0\}$ such that $1 - \varepsilon f \geq \kappa_0$, for any $s > 0$, there exists a constant $C^L_s(f)$, explicit and depending on $f$ only via $\rho$, $T$, $\kappa_0$, an upper bound on $\varepsilon$ and an upper bound on $H_{\varepsilon}(f)$ such that
\begin{equation} \label{eq:cercignanisoftlandau}
    \Deps^L (f)  \geq C^L_s(f) \, \left(\|f\|_{L^1_{s + 2}} + \|\nabla \sqrt{f}\|^2_{L^2_{1 +\frac{s}{2}}} \right)^{-\frac{\beta}{s}} \, H_{\varepsilon}[f|\Me]^{1 + \frac{\beta}{s}}.
\end{equation}
 
\medskip

 $\bullet$ \textbf{Hard potentials \emph{[}Adaptation of~\cite[Remark 1]{desvillettesLandauhardpotential}\emph{]}.}  Assume $\Psi(|z|) \geq |z|^{2 + \beta}$ with $\beta \in (0,1]$. Then, for any $\kappa_0 \in (0,1)$, $\varepsilon \geq 0$ and $0 \leq f \in L^1_2(\R^3) \setminus \{0\}$ such that $1 - \varepsilon f \geq \kappa_0$, there exist two explicit constants $K^L_1(f)$ and $K^L_2(f)$ which depend on $\beta$, on $f$ only via $\rho$, $T$, $\kappa_0$ and on an upper bound on $\|f\|^2_{L^2_6}$ such that
\begin{equation} \label{eq:cercignanihardlandau}
  \Deps^L (f) \leq K_1^L(f) \implies  \Deps^L (f)  \geq K^L_2(f) \, H_{\varepsilon}[f|\Me].
\end{equation}

\end{proposition}

\noindent We briefly discuss here the qualitative difference of our results with respect to the ones in~\cite{ABDLentropy,ABDLsoft} and~\cite{desvillettesLandauhardpotential}. The inequality we obtain in the over-Maxwellian case~\eqref{eq:cercignaniovermaxlandau} is almost identical to the one in~\cite[Theorem 1.4]{ABDLentropy}, with however an extra quantity in the constant in their inequality. As for the soft potentials case, the inequality in the proof of~\cite[Proposition 5.8]{ABDLsoft} involves the hard potential of order $\eta > 0$ entropy production (for some chosen $\eta>0$), which relates in~\eqref{eq:cercignanisoftlandau} with the $\|\nabla \sqrt{f}\|^2_{L^2_{1 +\frac{s}{2}}}$ norm (for some chosen $s>0$). We also point out that, while~\cite[Proposition 5.8]{ABDLsoft} holds for soft potentials of order at most $-\frac43$, there is no such restriction in~\eqref{eq:cercignanisoftlandau}. As far as hard potentials are concerned, the result we obtain (which could actually be more precise, but somewhat more complicated, by using ~\cite[Theorem 1]{desvillettesLandauhardpotential} instead of~\cite[Remark 1]{desvillettesLandauhardpotential}) is very similar to the one obtained by Desvillettes in~\cite[Proposition 2]{desvillettesLandauhardpotential}, with the difference that his result involves the relative Fisher information instead of the relative entropy (hence yielding a stronger result); our methodology only gives another proof to obtain it. Also, the inequality obtained in~\cite[Theorem 1.4]{ABDLentropy} comes with a more complicated constant than in~\cite[Proposition 2]{desvillettesLandauhardpotential} (similarly, than~\eqref{eq:cercignanihardlandau}), but holds without any assumption of smallness on $\Deps^L(f)$. We mention that a statement of the type~\eqref{eq:cercignanihardlandau} entails an exponential decay with explicit constants (if $K_1^L(f)$ and $K_2^L(f)$ are controlled below) as shown in~\cite[Lemma 1]{desvillettesLandauhardpotential}. Finally, we can also obtain inequalities for hard potentials of order $\beta > 1$ by adapting~\cite[Chapter~3,~Theorem~14,~(iii)]{villanireview}.

\begin{remark}
We want to mention~\cite{desvillettescoulomb} which bounds below the entropy production with a weighted version of the Fisher information, which would lead to a weighted relative entropy with power one in the right-hand-side of \eqref{eq:cercignanisoftlandau}. A similar result for hard potentials was proven in~\cite{carrapatoso}. We refer the interested reader to \cite{proceedingdesvilletteslandau} (in french) for a discussion on the entropic structure of the Landau operator. However our method of adapting results from the Classical to the Fermi-Dirac situation breaks here, as we did not prove any link between weighted entropies.
\end{remark}

\medskip
 
\noindent \textbf{Convergence towards equilibrium.} As an example of application of Proposition~\ref{theorem:LFDcerci}, we provide a convergence to equilibrium proposition in the hard potential case. We remind that this type of result was already obtained in~\cite{ABDLentropy} (and in~\cite{ABDLsoft} for soft potentials), although the constants that we obtain do differ, and our proof is shorter.

\begin{proposition} \textbf{Hard potentials case.}
Assume $\Psi(|z|) = |z|^{2 + \beta}$ with $\beta \in (0,1]$.
Consider $0 \leq f^{in} \in L^1_{s_{\beta}}(\R^3) \cap L^{\infty}(\R^3)$ with $s_{\beta} = \max \left( \frac{3 \beta}{2}, 4 - \beta \right)$. Then there exists $\e^{in} > 0$ depending only on $f^{in}$ such that for any $\e \in (0,\e^{in}]$, there exists a solution $f^{\e}$ to the homogeneous Landau-Fermi-Dirac equation (associated to the collision operator $Q_{\e}^L$ defined in~\eqref{eq:LFDoperator} and the initial distribution $f^{in}$) and two constants $C_1, C_2 > 0$ such that for any $t \geq 0$,
\begin{equation}
    \|f^{\e}(t, \cdot) - \M_{\e}^{f^{in}}\|_{L^1_2} \leq C_1 \, e^{-C_2 t}.
\end{equation}
\end{proposition}

\begin{proof}
The proof is similar to the one of Proposition~\ref{prop:ifthmBFD}, to which we refer the reader, and we only outline here the crucial points. First, we rely on~\cite[Corollary 3.7]{ABLorig} to obtain the existence of $\e^{in} > 0$ depending only on $f^{in}$ such that for any $\e \in (0,\e^{in}]$, the solution $f^{\e}$ associated to LFD$_{\e}$ satisfies 
$$
\sup_{t \geq 1} \; (1 - \e \|f^{\e}(t, \cdot)\|_{\infty}) \geq \frac12.
$$
We are now able to use~\eqref{eq:cercignanihardlandau}. In order for the proof to be complete, we must ensure that the coefficient that multiplies the relative entropy can be lower bounded uniformly in time. It amounts to upper bound the $L^2_6$ norm of $f^{\e}(t,\cdot)$ uniformly in time. This is possible as we have a uniform-in-time $L^{\infty}$ bound and~\cite[Theorem 1.3]{ABDLentropy} ensures the generation and uniform-in-time boundedness of all moments at some explicit time. Then the same argument as in~\cite[Lemma 1]{desvillettesLandauhardpotential} provides the announced exponential decay to equilibrium.
\end{proof}

\section{\texorpdfstring{$L^1$-$L^2$ weighted Csiszár-Kullback-Pinsker inequality}{L1L2 weighted Csiszár-Kullback-Pinsker inequality}}  \label{appendix:CK}

\noindent In this section, we discuss an optimized version of the $L^1$-$L^2$ weighted Csiszár-Kullback-Pinsker (CKP) inequality presented in Proposition~\ref{theorem:CKnonopti}. For a discussion on the CKP inequality and its previous generalizations, we refer the reader to the paragraph preceding Proposition~\ref{theorem:CKnonopti} in Subsection~\ref{subsection:CKP}. 

\medskip

\noindent We highlight that the idea of the proof of \eqref{eq:CKL1L2fermi} in the case $r=1$, $\varpi=1$, can be found in~\cite{lu2001spatially}, and the proof of \eqref{eq:CKL1L2be} (Bose-Einstein case) for $r=1$, $\varpi=1$ was also done with the same approach by Lu~\cite{lu2004isotropic}. Notice moreover that we consider here the case when $f$ is defined on $\R^3$, but \eqref{eq:CKL1L2classical}--\eqref{eq:CKL1L2fermi} and \eqref{eq:CKL1L2be} do hold in the broader setting where $f$ and $\varpi$ are defined on an arbitrary measured set, as long as the equilibrium distribution is well-defined. This is further detailed in Appendix~\ref{appendix:entropystudy}. Finally, we inform the reader that the Bose-Einstein case is briefly discussed in Appendix~\ref{appendix:boseeinstein}, where useful definitions may be found.

\begin{proposition} \textbf{$L^1$-$L^2$ weighted Csiszár-Kullback-Pinsker inequality.  \emph{[optimal]}} \label{theorem:CK}

\smallskip

\noindent Let $\varpi : \R^3 \to \R_+$ be measurable, and $r \in [1,2]$. Denote
\begin{equation} \label{eqdef:biglambda}
    \Lambda(\lambda) := \begin{cases}  
    \qquad \quad 2 &\text{ if } \lambda = 1,\\
    \displaystyle \frac{(\lambda-1)^2}{\lambda \log \lambda - \lambda + 1} \qquad &\text{ if } \lambda \in \R_+ \setminus\{1\}.
    \end{cases}
\end{equation}

\medskip

$\bullet$ \textbf{Classical CKP inequality.} For any $0\leq f \in L^1_2(\R^3) \cap L \log L (\R^3) \setminus \{0\}$, assuming that the norms below are finite,
\begin{equation} \label{eq:CKL1L2classical}
\|(f-M) \, \varpi\|^2_{L^r} \leq \left\|M \, \varpi^2 \right\|_{L^{\frac{r}{2-r}}} \Lambda \left(\frac{\left\|f \, \varpi^2 \right\|_{L^{\frac{r}{2-r}}}}{\left\|M \, \varpi^2 \right\|_{L^{\frac{r}{2-r}}}} \right)  H_{0}[f|M],
\end{equation}
where we denoted $M \equiv M^f$ the Maxwellian distribution associated to $f$ and $H_0$ is defined in \eqref{eqdef:entropy}--\eqref{eqdef:entropyrelat} (with $\varepsilon = 0$).

\medskip

$\bullet$ \textbf{Fermi-Dirac CKP inequality.} For any $\varepsilon > 0$ and $0\leq f \in L^1_2(\R^3) \setminus \{0\}$ such that $1 - \varepsilon f \geq 0$ and $\displaystyle \gamma > \frac25$, assuming that the norms below are finite,
\begin{equation} \label{eq:CKL1L2fermi}
\|(f-\M) \, \varpi\|^2_{L^r} \leq \left\|\M \, \varpi^2 \right\|_{L^{\frac{r}{2-r}}} \Lambda \left(\frac{\left\|f \, \varpi^2 \right\|_{L^{\frac{r}{2-r}}}}{\left\|\M \, \varpi^2 \right\|_{L^{\frac{r}{2-r}}}} \right)  H_{\varepsilon}[f|\M],
\end{equation}
where we denoted $\M \equiv \Me$ the $\varepsilon$-Fermi-Dirac distribution associated to $f$ and $H_{\varepsilon}$ is defined in \eqref{eqdef:entropy}--\eqref{eqdef:entropyrelat}.
\end{proposition}

\bigskip

\noindent The announced inequalities \eqref{eq:CKL1L2classicalnonopti}--\eqref{eq:CKL1L2ferminonopti} of Proposition~\ref{theorem:CKnonopti} are then consequence of Proposition~\ref{theorem:CK} and the fact that for any $a \geq 0$ and $b > 0$,
\begin{equation} \label{eq:Lambda2max}
b \, \Lambda \left( \frac{a}{b} \right) \leq 2 \, \max (a,b).
\end{equation}
The latter inequality can be deduced from the following ones. We claim that in fact,
\begin{alignat}{2}
\Lambda(\lambda) &\leq 1 + \lambda^{2/3}, \qquad \qquad  &&\lambda \in [0,1],  \label{ineq:Lbda1} \\
\Lambda(\lambda) &\leq   \frac23(2 + \lambda), \qquad && \lambda \in (1, 10) \label{ineq:Lbda2}\\ 
\Lambda(\lambda) &\leq   \frac{\lambda}{\log (\lambda) - 1}, \qquad && \lambda \geq 10. \label{ineq:Lbda3}
\end{alignat}
Notice the improvement of~\eqref{ineq:Lbda3} over \eqref{eq:Lambda2max} by a factor $\displaystyle \frac{1}{2 (\log a - \log b - 1)}$ when $\lambda = \frac{a}{b}$ is large, which, even if slowly, converges towards 0 as $a$ goes to infinity. The second inequality~\eqref{ineq:Lbda2} is well-known, and is actually the one used in the standard proof of the usual CKP inequality (see for instance~\cite{gilardoni}). As for the third one~\eqref{ineq:Lbda3}, we have for $\lambda \geq 10$,
$$
\Lambda(\lambda) = \frac{(\lambda-1)^2}{\lambda \log \lambda -\lambda + 1} =  \frac{\lambda-1}{\frac{\lambda}{\lambda-1} \log \lambda - 1} \leq \frac{\lambda}{\log (\lambda) - 1}.
$$
Finally, the proof of the first inequality~\eqref{ineq:Lbda1} is trickier and was suggested to me by Matthieu Dolbeault. Changing variables $x=\lambda^{1/3}$, it is enough to show that for any $x \in [0,1]$,
$$
(1 + x^2)(3x^3 \log x - x^3 + 1) - (1-x^3)^2 \geq 0.
$$
Developing, the above term actually equals $x^2 g(x)$, where
$$
g(x) = 1 + x - x^3 - x^4 + 3 (x + x^3) \log x.
$$
We can then conclude once we prove that $g$ is nonnegative on $(0,1)$. We compute its first four derivatives,
\begin{alignat*}{2}
    &g'(x) = 4 - 4 x^3 + (3 + 9 x^2) \log x, \qquad &&g''(x) = \frac{3}{x} + 9 x - 12 x^2 + 18 x \log x, \\
     &g'''(x) = 27 - \frac{3}{x^2} - 24 x + 18 \log x, \qquad &&g^{(4)}(x) = 6 (-4 + x^{-3} + 3 \, x^{-1}).
\end{alignat*}
The function $g^{(4)}$ is clearly nonnegative on $(0,1)$, hence $g'''$ is nondecreasing on $(0,1)$. But $g'''(1) = 0$, hence $g'''$ is nonpositive on $(0,1)$, thus $g''$ is nonincreasing on $(0,1)$. Again, $g''(1) = 0$ hence $g''$ is nonnegative, thus $g'$ is nondecreasing. Again, $g'(1) = 0$ hence $g'$ is nonpositive, thus $g$ is nonincreasing on $(0,1)$. Finally, $g(1) = 0$, so that $g$ is nonnegative on $(0,1)$.

\medskip

\noindent Let us now prove Proposition~\ref{theorem:CK}.

\begin{proof} This proposition is a consequence of the general inequality~\eqref{eq:realCKgeneral} in Corollary~\ref{cor:realCKgeneral} in Appendix~\ref{appendix:entropystudy}. We prove both inequalities~\eqref{eq:CKL1L2classical}--\eqref{eq:CKL1L2fermi} simultaneously, as the first one corresponds to the limit case $\varepsilon = 0$ of the second. Let us then consider $\varepsilon \geq 0$.

\medskip

\noindent We apply Corollary~\ref{cor:realCKgeneral} with $\Phi(x) \equiv \Phi_{\varepsilon}(x) =\int_0^x \log \varphi_{\varepsilon} (y) \, \dd y$, $J = (0,\varepsilon^{-1})$ and
$$
\F = \left\{0 \leq g \in L^1_2(\R^3) \left| \; 1 - \varepsilon g \geq 0, \quad  \int_{\R^3} g(v) \begin{pmatrix}
1 \\ v \\ |v|^2
\end{pmatrix} \dd v = \begin{pmatrix}
\rho \\ \rho \, u \\ 3 \, \rho T + \rho \, |u|^2
\end{pmatrix}\right. \right\},
$$
with  $\rho, T$ and $\varepsilon$ such that $\displaystyle \gamma > \frac{2}{5}$ (ensuring the existence of $\Me$). Then \eqref{eq:realCKgeneral} writes, since $H_{\Phi_{\varepsilon}} \equiv H_{\varepsilon}$ the $\varepsilon$-Fermi (or Classical in the case $\varepsilon = 0$) entropy, and $\displaystyle \frac{1}{\Phi_{\varepsilon}''}(x) = x(1-\varepsilon x) \leq x$,
$$
\left\| (f-\Me) \, \varpi \right\|_{L^r}^2 \leq \left(\int_0^1 (1-\tau) \, \left\|((1-\tau)\Me + \tau f) \, \varpi^2 \right\|_{L^{\frac{r}{2-r}}}^{-1} \, \dd \tau \right)^{-1} \, H_{\varepsilon}[f|\Me].
$$
We focus on the integral in the variable $\tau$. From Minkowski's inequality, we have
\begin{align*}
\left(\int_0^1 \frac{1-\tau}{\left\|((1-\tau)\Me + \tau f) \, \varpi^2 \right\|_{L^{\frac{r}{2-r}}}}   \, \dd \tau \right)^{-1} &\leq \left(\int_0^1 \frac{1-\tau}{ (1-\tau)\left\|\Me \, \varpi^2 \right\|_{L^{\frac{r}{2-r}}} + \tau \left\|f \, \varpi^2 \right\|_{L^{\frac{r}{2-r}}}} \, \dd \tau \right)^{-1} \\
&= \left\|\Me \, \varpi^2 \right\|_{L^{\frac{r}{2-r}}} \, \Lambda \left( \frac{\left\|f \, \varpi^2 \right\|_{L^{\frac{r}{2-r}}}}{\left\|\Me \, \varpi^2 \right\|_{L^{\frac{r}{2-r}}}}\right),
\end{align*}
where $\Lambda$ is defined in \eqref{eqdef:biglambda}. Indeed, this last equality comes from the Taylor expansion of $\lambda \mapsto \lambda \log \lambda - \lambda$ around $1$,
$$
\lambda \log \lambda - \lambda = -1 + (\lambda-1) \, \log 1 + (\lambda-1)^2 \int_0^1 \frac{1-\tau}{ (1-\tau) + \tau \lambda} \, \dd \tau,
$$
that is
\begin{align*}
\int_0^1 \frac{1-\tau}{ (1-\tau) + \tau \lambda} \, \dd \tau &= \begin{cases}  
    \qquad \quad \frac12 &\text{ if } \lambda = 1,\\
    \displaystyle \frac{\lambda \log \lambda - \lambda + 1}{(\lambda-1)^2} \qquad &\text{ if } \lambda \in \R_+ \setminus\{1\}
    \end{cases} \\
&= \frac{1}{\Lambda(\lambda)}.
\end{align*}
\end{proof}

\bigskip

\noindent Finally, as a corollary to Proposition~\ref{theorem:CK} come the standard Csiszár-Kullback-Pinsker inequalities.

\begin{corollary} \textbf{Standard Csiszár-Kullback-Pinsker inequalities. } \label{corollary:CK}

\medskip

\noindent For any real number $x \in \R$, we denote in the following $x_+ = \max(0,x)$.

\medskip

$\bullet$ \textbf{Standard Classical CKP inequalities.} For any $\alpha \geq 0$ and $0\leq f \in L^1_2(\R^3) \cap L \log L (\R^3) \setminus \{0\}$,
\begin{alignat}{3}
&\|f-M \|^2_{L^1} &&\leq 2 \, \left\|M  \right\|_{L^{1}} \,   &&H_{0}[f|M],  \label{eq:usualCKclassical1}\\
&\|(M-f)_+ \|^2_{L^1_{\alpha}} &&\leq 2 \, \left\|M  \right\|_{L^{1}_{2 \alpha}} \,  &&H_{0}[f|M], \label{eq:usualCKclassical2} \\
&\|f-M \|^2_{L^1_2} &&\leq 8 \, \left\|M  \right\|_{L^1_4} \,  &&H_{0}[f|M], \label{eq:usualCKclassical3} 
\end{alignat}
where we denoted $M \equiv M^f$ the Maxwellian distribution associated to $f$.

\medskip

$\bullet$ \textbf{Standard Fermi-Dirac CKP inequalities.} For any $\varepsilon > 0$, $\alpha \geq 0$ and $0\leq f \in L^1_2(\R^3) \setminus \{0\}$ such that $1 - \varepsilon f \geq 0$ and $\gamma > \frac25$,
\begin{alignat}{3}
&\|f-\M \|^2_{L^1} &&\leq 2 \, \left\|\M  \right\|_{L^{1}} \,   &&H_{\varepsilon}[f|\M],  \label{eq:usualCKfermi1}\\
&\|(\M-f)_+ \|^2_{L^1_{\alpha}} &&\leq 2 \, \left\|\M  \right\|_{L^{1}_{2 \alpha}} \,  &&H_{\varepsilon}[f|\M], \label{eq:usualCKfermi2} \\
&\|f-\M \|^2_{L^1_2} &&\leq 8 \, \left\| \M  \right\|_{L^1_4} \,  &&H_{\varepsilon}[f|\M],   \label{eq:usualCKfermi3}
\end{alignat}
where we denoted $\M \equiv \Me$ the $\varepsilon$-Fermi-Dirac distribution associated to $f$.
\end{corollary}

\begin{proof}  
We only prove \eqref{eq:usualCKfermi1}--\eqref{eq:usualCKfermi3} as \eqref{eq:usualCKclassical1}--\eqref{eq:usualCKclassical3} can be seen as a limit case of the previous ones when $\varepsilon \to 0$.

\medskip

$\bullet$ \emph{Proof of  \eqref{eq:usualCKfermi1}.} Applying Proposition~\ref{theorem:CK}, specifically Equation \eqref{eq:CKL1L2fermi}, with $r = 1$ and $\varpi = 1$, we obtain
$$
\|f- \M \|^2_{L^1} \leq \left\|\M  \right\|_{L^{1}} \, \Lambda \left( \frac{\left\|f  \right\|_{L^{1}}}{\left\|\M  \right\|_{L^{1}}} \right)  H_{\varepsilon}[f|\M].
$$
Since $\left\|f  \right\|_{L^{1}} = \left\|M  \right\|_{L^{1}}$ and $\Lambda(1) = 2$, we obtain \eqref{eq:usualCKfermi1}.

\medskip

$\bullet$ \emph{Proof of \eqref{eq:usualCKfermi2}.} Applying Proposition~\ref{theorem:CK}, specifically Equation \eqref{eq:CKL1L2fermi}, with $\varpi(v) = (1 + |v|^2)^{\alpha /2} \, \mathbf{1}_{f \leq \M}$ and $r = 1$, we obtain
$$
\|(\M-f)_+ \|^2_{L^1_{\alpha}} \leq \|\M \, \mathbf{1}_{f \leq \M} \|_{L^1_{2 \alpha}} \, \Lambda \left( \frac{\|f \, \mathbf{1}_{f \leq \M} \|_{L^1_{2 \alpha}}}{\|\M \, \mathbf{1}_{f \leq \M} \|_{L^1_{2 \alpha}}} \right)  H_{\varepsilon}[f|\M].
$$
Since $\displaystyle \|f \, \mathbf{1}_{f \leq \M} \|_{L^1_{2 \alpha}} \leq \|\M \, \mathbf{1}_{f \leq \M} \|_{L^1_{2 \alpha}}$ and, from \eqref{ineq:Lbda1}, $\Lambda \leq 2$ on $[0,1]$, we obtain \eqref{eq:usualCKfermi2} after upper bounding $\|\M \, \mathbf{1}_{f \leq \M} \|_{L^1_{2 \alpha}}$ by $\|\M \|_{L^1_{2 \alpha}}$.

\medskip

$\bullet$ \emph{Proof of \eqref{eq:usualCKfermi3}.} We remark that $|f-\M| = f-\M + 2 (\M-f)_+$, and
$$
\int_{\R^3} (f-\M)(v) \, (1 + |v|^2) \, \dd v = 0,
$$
so that
$$
\|f-\M\|^2_{L^1_2} = 4 \|(\M-f)_+\|^2_{L^1_2}.
$$
We apply \eqref{eq:usualCKfermi2} with $\alpha = 2$ and obtain \eqref{eq:usualCKfermi3}.
\end{proof}

\appendix

\section{Similar results in the Bose-Einstein case} \label{appendix:boseeinstein}

\noindent An upper-bound inequality similar to \eqref{eq:theorementropiesFD} can also be obtained in the Bose-Einstein case. This latter case formally corresponds to taking $- \varepsilon$ instead of $\varepsilon$ in our formulas. First define for any $x \in \R_+$ and $\varepsilon > 0$,
$$
\varphi^{BE}_{\varepsilon}(x) := \frac{x}{1 + \varepsilon x}, \qquad \Phi_{\varepsilon}^{BE}(x) := \int_0^x \log \varphi^{BE}_{\varepsilon}(y) \, \dd y,
$$
and the Bose-Einstein entropy of $0 \leq f \in L^1_2(\R^3)$:
$$
H^{BE}_{\varepsilon}(f) := \int_{\R^3} \Phi_{\varepsilon}^{BE}(f) \, \dd v.
$$
Lu proved in~\cite{luBE} that, under the condition
\begin{equation} \label{eqassump:TTc}
T \geq \frac{\zeta(\frac52)}{\zeta(\frac32)} \, T_c, \qquad \quad T_c := \frac{1}{2 \pi} \left( \frac{\rho \, \varepsilon}{\zeta(\frac32)} \right)^{2/3},
\end{equation}
where $\zeta$ is the Riemann Zêta function and $T_c$ is called the critical temperature, there exists a unique $\varepsilon$-Bose-Einstein statistics $\M_{\varepsilon}^{BE,f}$ associated to $f$ - that is a distribution such that $\log \varphi^{BE}_{\varepsilon} (\M_{\varepsilon}^{BE,f})$ is a linear combination of $v\mapsto 1$, $v\mapsto v$ and $v\mapsto |v|^2$ and sharing the same normalization in $v\mapsto 1$, $v\mapsto v$ and $v\mapsto |v|^2$ as $f$. We can obtain from Proposition~\ref{prop:general} in Appendix~\ref{appendix:entropystudy} that,  denoting $H^{BE}_{\varepsilon}[f|\M_{\varepsilon}^{BE,f}] := H^{BE}_{\varepsilon}(f) - H^{BE}_{\varepsilon}(\M_{\varepsilon}^{BE,f})$, we have
\begin{align} 
    H^{BE}_{\varepsilon}[f|\M_{\varepsilon}^{BE,f}] &= \int_0^1 (1-\tau) \left( \int_{\R^3} \left(f(v)-\M_{\varepsilon}^{BE,f}(v) \right)^2 \, {\Phi^{BE}_{\varepsilon}}'' \left((1-\tau)\M_{\varepsilon}^{BE,f}(v) + \tau f(v) \right) \, \dd v \right) \dd \tau \nonumber \\
    &= \int_{\R^3} \int_{\M_{\varepsilon}^{BE,f}(v)}^{f(v)}  \frac{f(v) - x}{\varphi^{BE}_{\varepsilon}(x)} \, {\varphi^{BE}_{\varepsilon}}'(x) \, \dd x \, \dd v, \label{eq:firstformulaHepsBE}
\end{align}
where the last equality comes from ${\Phi^{BE}_{\varepsilon}}'' = \frac{{\varphi^{BE}_{\varepsilon}}'}{\varphi^{BE}_{\varepsilon}}$ and the change of variables $x = \M_{\varepsilon}^{BE,f}(v) + \tau (f(v) - \M_{\varepsilon}^{BE,f})$.

\medskip

\noindent In the following Proposition, we provide a link between the relative entropies to equilibrium of the Bose-Einstein and the classical cases. Although I believe that both inequalities could be obtained, we only present here the ``upper-bound'' inequality as its proof is rather short. Further work may allow to obtain the lower-bound inequality, with a constant that probably depends on an $L^{\infty}$ bound on $f$. This constitutes another reason why we did not investigate this other inequality, as, although $L^{\infty}$ bounds are natural to use in the Fermi-Dirac context, due to Pauli's exclusion principle, they are not in the Bose-Einstein one, due to the phenomenon of condensation.

\begin{proposition} \label{prop:BE}
\textbf{Upper-bound in the Bose-Einstein case.} For any $\varepsilon > 0$ and nonnegative $f \in L^1_2(\R^3) \cap L \log L (\R^3) \setminus \{0\}$ which density and temperature satisfy \eqref{eqassump:TTc}, we have
\begin{equation} \label{ineq:BE}
    H_0\left[ \left. \frac{f}{1 + \varepsilon f} \right|M^{\frac{f}{1 + \varepsilon f}} \right] \leq  H^{BE}_{\varepsilon} \left[f|\M_{\varepsilon}^{BE,f} \right].
\end{equation}
\end{proposition}
\begin{proof}
The starting point of the proof are Equation \eqref{eq:firstformulaHepsBE} and the inequality $|y-z| \geq |\varphi^{BE}_{\varepsilon}(y) - \varphi^{BE}_{\varepsilon}(z)|$ for all $(y,z) \in \R_+^2$, yielding
$$
H^{BE}_{\varepsilon} \left[f|\M_{\varepsilon}^{BE,f} \right] =  \int_{\R^3} \int_{\M_{\varepsilon}^{BE,f}}^f  \frac{f - x}{\varphi^{BE}_{\varepsilon}(x)} \, {\varphi^{BE}_{\varepsilon}}'(x) \, \dd x \, \dd v \geq  \int_{\R^3} \int_{\M_{\varepsilon}^{BE,f}}^f  \frac{\varphi^{BE}_{\varepsilon}(f) - \varphi^{BE}_{\varepsilon}(x)}{\varphi^{BE}_{\varepsilon}(x)} \, {\varphi^{BE}_{\varepsilon}}'(x) \, \dd x \, \dd v.
 $$
 Applying the change of variables $y = \varphi^{BE}_{\varepsilon}(x)$ and using formula \eqref{eq:firstformulaH01}, we obtain
 $$
H^{BE}_{\varepsilon} \left[f|\M_{\varepsilon}^{BE,f} \right] \geq  H_0\left[ \varphi^{BE}_{\varepsilon}(f)|M^{\varphi^{BE}_{\varepsilon}(f)} \right] + \int_{\R^3} \int_{\varphi^{BE}_{\varepsilon}(\M_{\varepsilon}^{BE,f}) }^{M^{\varphi^{BE}_{\varepsilon}(f)} } \frac{\varphi^{BE}_{\varepsilon}(f)  - y}{y}  \, \dd y \, \dd v.
$$  
Remark that, as $\log M^{\varphi^{BE}_{\varepsilon}(f)} - \log \varphi^{BE}_{\varepsilon}(\M_{\varepsilon}^{BE,f})$ is a linear combination of conserved quantities (namely, $v \mapsto 1$, $v \mapsto v$ and $v \mapsto |v|^2$), we have by definition of $M^{\varphi^{BE}_{\varepsilon}(f)}$ that
$$
\int_{\R^3} \varphi^{BE}_{\varepsilon}(f) \, \log \left( \frac{M^{\varphi^{BE}_{\varepsilon}(f)}}{\varphi^{BE}_{\varepsilon}(\M_{\varepsilon}^{BE,f})}\right) \dd v = \int_{\R^3} M^{\varphi^{BE}_{\varepsilon}(f)} \, \log \left( \frac{M^{\varphi^{BE}_{\varepsilon}(f)}}{\varphi^{BE}_{\varepsilon}(\M_{\varepsilon}^{BE,f})}\right) \dd v,
$$  
so that
$$
\int_{\R^3} \int_{\varphi^{BE}_{\varepsilon}(\M_{\varepsilon}^{BE,f}) }^{M^{\varphi^{BE}_{\varepsilon}(f)} } \frac{\varphi^{BE}_{\varepsilon}(f)  - y}{y}  \, \dd y \, \dd v = \int_{\R^3} \int_{\varphi^{BE}_{\varepsilon}(\M_{\varepsilon}^{BE,f}) }^{M^{\varphi^{BE}_{\varepsilon}(f)} } \frac{M^{\varphi^{BE}_{\varepsilon}(f)}  - y}{y}  \, \dd y \, \dd v \geq 0,
$$
ending the proof.
\end{proof}

\begin{proposition} \label{prop:CKBE}
 \textbf{Bose-Einstein CKP inequality.} Let $\varpi : \R^3 \to \R_+$ be measurable, and $r \in [1,2]$. We recall the definition~\eqref{eqdef:biglambda} of the function $\Lambda$,
\begin{equation*} 
    \Lambda(\lambda) := \begin{cases}  
    \qquad \quad 2 &\text{ if } \lambda = 1,\\
    \displaystyle \frac{(\lambda-1)^2}{\lambda \log \lambda - \lambda + 1} \qquad &\text{ if } \lambda \in \R_+ \setminus\{1\}.
    \end{cases}
\end{equation*} 
Then for any $\varepsilon > 0$ and $0\leq f \in L^1_2(\R^3) \cap L \log L (\R^3)  \setminus \{0\}$ which density and temperature satisfy \eqref{eqassump:TTc}, assuming that the norms below are finite,
\begin{equation} \label{eq:CKL1L2be}
\left\| \left(f-\M \right) \, \varpi \right\|^2_{L^r} \leq \left\|\M(1 + \varepsilon \M) \, \varpi^2 \right\|_{L^{\frac{r}{2-r}}} \Lambda \left(\frac{\left\|f(1 + \varepsilon f) \, \varpi^2 \right\|_{L^{\frac{r}{2-r}}}}{\left\|\M(1 + \varepsilon \M) \, \varpi^2 \right\|_{L^{\frac{r}{2-r}}}} \right)  H^{BE}_{\varepsilon} \left[ f \left|\M \right. \right],
\end{equation}
where we denoted $\M \equiv \M_{\varepsilon}^{BE,f}$ the $\varepsilon$-Bose-Einstein distribution associated to $f$, and $H^{BE}_{\varepsilon}$ is the Bose-Einstein entropy. When $r=2$, $L^{\frac{r}{2-r}}$ shall be understood as $L^{\infty}$.
\end{proposition}

\begin{proof}
The proof is similar to the one of Proposition~\ref{theorem:CK}. Let $\varepsilon > 0$.

\medskip

\noindent We apply Corollary~\ref{cor:realCKgeneral} with $\Phi(x) \equiv \Phi^{BE}_{\varepsilon}(x) = \int_0^x \log \frac{y}{1 + \varepsilon y} \, \dd y$, $J = \R_+^*$ and
$$
\F = \left\{0 \leq g \in L^1_2(\R^3) \left| \;   \int_{\R^3} g(v) \begin{pmatrix}
1 \\ v \\ |v|^2
\end{pmatrix} \dd v = \begin{pmatrix}
\rho \\ \rho \, u \\ 3 \, \rho T + \rho \, |u|^2
\end{pmatrix}\right. \right\},
$$
with  $\rho, T$ and $\varepsilon$ such that \eqref{eqassump:TTc} is satisfied, ensuring the existence of $\M_{\varepsilon}^{BE,f}$. In the rest of this proof, we denote for the sake of clarity $\M \equiv \M_{\varepsilon}^{BE,f}$. Then \eqref{eq:realCKgeneral} writes, since $H_{\Phi^{BE}_{\varepsilon}} \equiv H^{BE}_{\varepsilon}$ the $\varepsilon$-Bose-Einstein entropy, and $\displaystyle \frac{1}{{\Phi^{BE}_{\varepsilon}}''}(x) = x(1+\varepsilon x)$ is convex,
$$
\left\| (f-\M) \, \varpi \right\|_{L^r}^2 \leq \left(\int_0^1 (1-\tau) \, \left\|((1-\tau)\M(1 + \varepsilon \M) + \tau f(1 + \varepsilon f)) \, \varpi^2 \right\|_{L^{\frac{r}{2-r}}}^{-1} \, \dd \tau \right)^{-1} \, H^{BE}_{\varepsilon}[f|\M].
$$
We focus on the term with the integral in $\tau$. From Minkowski's inequality, it is smaller than
\begin{align*}
&\left(\int_0^1 \frac{1-\tau}{ (1-\tau)\left\|\M(1 + \varepsilon \M) \, \varpi^2 \right\|_{L^{\frac{r}{2-r}}} + \tau \left\|f(1 + \varepsilon f) \, \varpi^2 \right\|_{L^{\frac{r}{2-r}}}} \, \dd \tau \right)^{-1} \\
= \; &\left\|\M(1 + \varepsilon \M) \, \varpi^2 \right\|_{L^{\frac{r}{2-r}}} \, \Lambda \left( \frac{\left\|f(1 + \varepsilon f) \, \varpi^2 \right\|_{L^{\frac{r}{2-r}}}}{\left\|\M(1 + \varepsilon \M) \, \varpi^2 \right\|_{L^{\frac{r}{2-r}}}}\right),
\end{align*}
where $\Lambda$ is defined in~\eqref{eqdef:biglambda} and appears thanks to a Taylor expansion of $\lambda \mapsto \lambda \log \lambda - \lambda$ around 1 like in the proof of Proposition~\ref{theorem:CK}, allowing to conclude.
\end{proof}

\noindent From the previous proposition, we easily deduce the following standard inequalities.
\begin{corollary}
\textbf{Standard Bose-Einstein CKP inequalities.} For any $\varepsilon > 0$, $\alpha \geq 0$ and $0\leq f \in L^1_2(\R^3) \cap L \log L (\R^3)  \setminus \{0\}$ satisfying \eqref{eqassump:TTc},
\begin{alignat}{3} 
&\|(\M-f)_+ \|^2_{L^1_{\alpha}} &&\leq 2 \left( \left\|\M  \right\|_{L^{1}_{2 \alpha}} + \varepsilon  \left\|\M  \right\|^2_{L^{2}_{\alpha}} \right) \,  &&H^{BE}_{\varepsilon}[f|\M], \label{eq:usualCKbe1} \\
&\|f-\M \|^2_{L^1} &&\leq 8 \, \left(\left\|\M  \right\|_{L^{1}} + \varepsilon  \|\M\|_{L^2}^2 \right)   &&H^{BE}_{\varepsilon}[f|\M],  \label{eq:usualCKbe2} \\
&\|f-\M \|^2_{L^1_2} &&\leq 8 \, \left( \left\|\M  \right\|_{L^{1}_{4}} + \varepsilon  \left\|\M  \right\|^2_{L^{2}_{2}} \right) \,  &&H^{BE}_{\varepsilon}[f|\M],   \label{eq:usualCKbe3}
\end{alignat}
where we denoted for clarity $\M\equiv \M_{\varepsilon}^{BE,f}$ the $\varepsilon$-Bose-Einstein distribution associated to $f$.
\end{corollary}
\noindent We prove the above inequalities similarly as we did for Corollary~\ref{corollary:CK}. For~\eqref{eq:usualCKbe1} we apply Proposition~\ref{prop:CKBE} with $r = 1$ and $\varpi(v) = (1+|v|^2)^{\frac{\alpha}{2}} \, \mathbf{1}_{f \leq \M}$ and notice that $\Lambda \leq 2$ on $[0,1]$. We then obtain~\eqref{eq:usualCKbe2}--\eqref{eq:usualCKbe3} by decomposing $|f-\M| = f - \M + 2(\M-f)_+$, using the fact that $f$ and $\M$ share the same normalization in $v\mapsto 1$, $v\mapsto v$ and $v\mapsto |v|^2$, and~\eqref{eq:usualCKbe1} with respectively $\alpha = 0$ and $\alpha=2$.

\section{A general discussion about entropies and equilibria} \label{appendix:entropystudy}

\noindent In this section we intend to provide general considerations on the entropy, which are much more general than the scope of this paper, but give a good understanding of the notions we used, and could also be helpful in the study of weak turbulence, where various kinds of unusual entropies can emerge (see~\cite{bredendesvillettes}). Our setting is laid down quite generally. Consider a measured space $(\E,\A,\mu)$, an open interval $J \subset \R$ which closure we denote by $\Bar{J}$, and a function $\Phi \in \mathcal{C}^2(J) \cap \mathcal{C}^0(\Bar{J})$ such that $\Phi'' > 0$ on $J$. Remark that $\Phi'$ is then a $\mathcal{C}^1$-diffeomorphism from $J$ onto $\Phi'(J)$.

\medskip

\noindent \textbf{Entropy.} We define the $\Phi$-entropy, for any ($\A$, Bor$(\Bar{J})$)-measurable $f : \E \to \Bar{J}$ such that the following integral makes sense and is finite, by
\begin{equation} \label{eqdef:philantrop}
H_{\Phi}(f) := \int_{\E} \Phi(f(\zeta)) \, \dd \mu(\zeta),
\end{equation}
and we denote by $E_{\Phi}$ the set of such $f$. We let the relative $\Phi$-entropy of $f$ and $g$ to be
\begin{equation} \label{eqdef:relatphilantrop}
H_{\Phi}[f|g] = H_{\Phi}(f) - H_{\Phi}(g).
\end{equation}
We also define, for ($\A$,~Bor$(\Bar{J})$)-measurable $f,g : \E \to \Bar{J}$, the $\Phi$-relative-entropy (which in general differs from the relative $\Phi$-entropy) of $f$ and $g$ by
\begin{equation} \label{eq:phirelatentrop}
\HH_{\Phi}[f|g] := \int_0^1 (1-\tau) \left( \int_{g \in J} (f-g)^2 \, \Phi''((1-\tau)g + \tau f) \, \dd \mu(\zeta) \right) \dd \tau.
\end{equation}
Note that $\HH_{\Phi}[f|g]$ is possibly infinite, but always well-defined, as
$$
g \in J \implies \; \forall \, \tau \in (0,1), \; \; (1-\tau)g + \tau f \in J,
$$
and that $\HH_{\Phi}[f|g]$ is always nonnegative. The following Proposition~\ref{prop:general} gives a quite simple but general result linking entropies, conserved quantities and equilibrium distributions, under a sole existence assumption.

\begin{proposition} \label{prop:general}
Let $I$ be a countable set, $(\phi_i)_{i \in I}$ a family of measurable real functions, $(\omega_i)_{i \in I}$ a family of real numbers, and
$$
\F = \left\{f \in E_{\Phi} \text{ s.t. } \forall \, i \in I, \; \; \int_{\E} |f(\zeta)| \, |\phi_i(\zeta)| \, \dd \mu(\zeta) < \infty \text{ and } \int_{\E} f(\zeta) \, \phi_i(\zeta) \, \dd \mu(\zeta) = \omega_i \right\}.
$$
Assume that $\Phi'(J) = \R$, and that there exists $(\alpha_i(\omega)) \in \R^I$ such that    $M^{\F}_{\Phi} \in \F$, where
\begin{equation} \label{eq:generalequilibrium}
    M^{\F}_{\Phi} := (\Phi')^{-1} \left(\sum_{i \in I}\alpha_i(\omega) \, \phi_i \right). 
\end{equation}
Then the following four propositions are equivalent. Let $g \in \F$.
\begin{align*}
   (i) \; &g \in J \quad \mu\text{-a.e.} \; \text{ and } \; \;   \forall \, f \in \F, \;  H_{\Phi}[f|g] = \HH_{\Phi}[f|g], \\
   (ii) \; &g \in J \quad \mu\text{-a.e.} \; \text{ and } \; \;  \forall \, f \in \F, \; \; \HH_{\Phi}[f|g] < \infty \; \text{ and } \; \;   \int_{\E} (f-g) \, \Phi'(g) \, \dd \mu (\zeta) = 0, \\
   (iii) \; &H_{\Phi}(g) = \underset{h \in \F}{\min} \;  H_{\Phi}(h), \\
   (iv) \; &g = M^{\F}_{\Phi} \quad \mu \text{-a.e.}
\end{align*} 
In particular, $M^{\F}_{\Phi}$ is the unique minimizer of $H_{\Phi}$ under the constraints of the set $\F$.

\end{proposition}
 
\noindent The above proposition actually proves the following (under the assumptions $\Phi'(J) = \R$ and of existence of $M^{\F}_{\Phi}$). An admissible distribution $g$ is an equilibrium relative to the conserved quantities $\phi_i$, in the sense of the minimization of the $\Phi$-entropy $[(iii)]$, if and only if $\Phi'(g)$ is a linear combination of the functions $\phi_i$ $[(iv)]$, if and only if the relative $\Phi$-entropy between any admissible distribution $f$ and $g$ is given by \eqref{eq:phirelatentrop} $[(i)]$, if and only if the quantity $\Phi'(g)$ is conserved amongst all admissible distributions $[(ii)]$ - indeed, $(ii)$ is equivalent, assuming the following integrals make sense, to $\forall \, f_1,f_2 \in \F$, $\int_{\E} f_1 \, \Phi'(g) \, \dd \mu (\zeta) = \int_{\E} f_2 \, \Phi'(g) \, \dd \mu (\zeta)$. 

\smallskip

\noindent The reader may notice that the Classical case corresponds to the choice $\Phi(x) \equiv \Phi_0(x) = x \log x - x$, for which $(\Phi'_0)^{-1} = \exp$, hence $M^{\F}_{\Phi}$ is in this case a Maxwellian, since the conserved quantities (corresponding to the functions $\phi_i$ in the proposition) are $v \mapsto 1$, $v \mapsto v$ and $v \mapsto |v|^2$. 

\smallskip

\noindent Moreover, the Fermi-Dirac case corresponds to the choice $\displaystyle \Phi(x) \equiv \Phi_{\varepsilon}(x) \equiv \int_0^x \log \frac{y}{1 - \varepsilon y} \, \dd y$, for which $\displaystyle (\Phi'_{\varepsilon})^{-1}(x) = \frac{e^x}{1 + \varepsilon e^x}$, hence $M^{\F}_{\Phi}$ is in this case a Fermi-Dirac distribution, since again, the conserved quantities (corresponding to the functions $\phi_i$ in the proposition) are $v \mapsto 1$, $v \mapsto v$ and $v \mapsto |v|^2$. Proposition~\ref{prop:firstformulaHeps} then comes as a corollary to Proposition~\ref{prop:general} with $\E = \R^3$ endowed with the Lebesgue measure, $J = (0,\varepsilon^{-1})$, $\displaystyle \Phi(x) \equiv \Phi_{\varepsilon}(x) \equiv \int_0^x \log \varphi_{\varepsilon}(y) \, \dd y$, $M^{\F}_{\Phi} \equiv \Me$ and
$$
\F = \left\{0 \leq f \in L^1_2(\R^3) \left| \; 1 - \varepsilon f \geq 0, \quad \int_{\R^3} f(v) \begin{pmatrix}
1 \\ v \\ |v|^2
\end{pmatrix} \dd v = \begin{pmatrix}
\rho \\ \rho \, u \\ 3 \, \rho T + \rho \, |u|^2
\end{pmatrix}\right. \right\},
$$
with  $\rho, T$ and $\varepsilon$ such that $\displaystyle \gamma > \frac{2}{5}$ (ensuring the existence of $\Me$, as proven in~\cite{lu2001spatially}).
\begin{remark}
    The Bose-Einstein case is also recovered with $\displaystyle \Phi(x) \equiv \Phi^{BE}_{\varepsilon}(x) \equiv \int_0^x \log \frac{y}{1 + \varepsilon y} \, \dd y$, for which $\displaystyle (\Phi'_{-\varepsilon})^{-1}(x) = \frac{e^x}{1 - \varepsilon e^x}$, hence $M^{\F}_{\Phi}$ is in this case a Bose-Einstein distribution (when it exists).
\end{remark} 
\begin{proof}
We start by proving $(i) \iff (ii)$. Let $f,g \in \F$ such that $g \in J$  $\mu$-almost everywhere and $\HH_{\Phi}[f|g]~<~\infty$. Using a Taylor expansion followed by Fubini's Theorem, we get
\begin{align*}
 H_{\Phi}[f|g] &= \int_{\E} (\Phi(f) - \Phi(g)) \, \dd \mu(\zeta) = \int_{g \in J} (\Phi(f) - \Phi(g)) \, \dd \mu(\zeta) \\
 &= \int_{g \in J} \left( (f-g) \Phi'(g) + (f-g)^2 \int_0^1 (1-\tau) \Phi''((1-\tau)g + \tau f) \, \dd \tau  \right) \dd \mu (\zeta) \\
 &= \int_{g \in J} (f-g) \Phi'(g) \, \dd \mu (\zeta) + \HH_{\Phi}[f|g],
\end{align*}
thus proving the announced equivalence. Remark that in the last equality we used the fact that $\HH_{\Phi}[f|g] < \infty$ to ensure that $\int_{g \in J} (f-g) \Phi'(g) \, \dd \mu (\zeta)$ is well-defined. We now show $(iv) \implies (ii)$. Since $\mathrm{Im}({\Phi'}^{-1}) = J$, we do have $M^{\F}_{\Phi} \in J$  $\mu$-almost everywhere, and
\begin{align*}
\int_{M^{\F}_{\Phi} \in J} (f-M^{\F}_{\Phi}) \, \Phi'(M^{\F}_{\Phi}) \, \dd \mu (\zeta) &= \int_{\E}(f-M^{\F}_{\Phi}) \sum_{i} \alpha_i(\omega_i) \, \phi_i \, \dd \mu(\zeta) \\&= \sum_i \alpha_i(\omega_i) \left(\int_{\E} f \, \phi_i \, \dd \mu(\zeta) - \int_{\E} M^{\F}_{\Phi} \, \phi_i \, \dd \mu(\zeta) \right) = 0,
\end{align*}
where the last equality comes from the fact that both $f$ and $M^{\F}_{\Phi}$ belong to $\F$. We now focus on $(ii) \implies (iv)$. Assume the existence of $g \in \F$ such that $g \in J$  $\mu$-almost everywhere and
$$
\forall \, f \in \F, \quad \int_{g \in J} (f-g) \, \Phi'(g) \, \dd \mu(\zeta) = 0.
$$
Since we also have $M^{\F}_{\Phi} \in\F$, then $M^{\F}_{\Phi} \in J$  $\mu$-almost everywhere and we just proved that
$$
\int_{M^{\F}_{\Phi} \in J} (g-M^{\F}_{\Phi}) \, \Phi' (M^{\F}_{\Phi}) \, \dd \mu(\zeta) = 0,
$$
allowing to deduce that
$$
\int_{\E}(g-M^{\F}_{\Phi})(\Phi'(g) - \Phi'(M^{\F}_{\Phi})) \, \dd \mu(\zeta) = 0.
$$
Since $\Phi'$ is increasing, this implies that $g = M^{\F}_{\Phi}$ $\mu$-almost everywhere. We now focus on $(iv)~\implies~(iii)$. Since we already proved $(iv) \implies (ii) \implies (i)$, we have for any $f \in \F$,
$$
H_{\Phi}[f|M^{\F}_{\Phi}] = \HH_{\Phi}[f|M^{\F}_{\Phi}] \geq 0,
$$
thus
$$
H_{\Phi}(f) \geq H_{\Phi}(M^{\F}_{\Phi}).
$$
Finally, we prove $(iii) \implies (iv)$. Assume $H_{\Phi}(g) = \underset{h \in \F}{\min} \;  H_{\Phi}(h)$. We just proved that $H_{\Phi}(M^{\F}_{\Phi}) = \underset{h \in \F}{\min} \;  H_{\Phi}(h)$, hence $H_{\Phi}(g) = H_{\Phi}(M^{\F}_{\Phi})$ and $H_{\Phi}[g|M^{\F}_{\Phi}] = 0$. Since we also proved $(iv) \implies (i)$, we know that 
$$
\HH_{\Phi}[g|M^{\F}_{\Phi}] = H_{\Phi}[g|M^{\F}_{\Phi}],
$$
hence $\HH_{\Phi}[g|M^{\F}_{\Phi}] = 0$, that is
$$
\int_0^1 (1-\tau) \left( \int_{M^{\F}_{\Phi} \in J} (g-M^{\F}_{\Phi})^2 \, \Phi''((1-\tau)M + \tau g) \, \dd \mu(\zeta) \right) \dd \tau = 0.
$$
This implies, since $M^{\F}_{\Phi} \in J$ and $(1-\tau)M^{\F}_{\Phi} + \tau g \in J$ $\mu$-almost everywhere for any $\tau \in (0,1)$, and $\Phi''>0$ on $J$, that $g=M^{\F}_{\Phi}$ $\mu$ almost everywhere.
\end{proof}

\medskip

\noindent \textbf{General Csiszár-Kullback-Pinsker inequality.} The famous Csiszár-Kullback-Pinsker inequality, linking the squared $L^1$ distance of two probabilities with their relative Classical entropy can in fact be generalized to the whole family of $\Phi$-entropies, and for weighted $L^r$, $1 \leq r \leq 2$ distances, by the following Proposition~\ref{theorem:generalCK}, and more specifically, in the context of convergence towards equilibrium, by Corollary~\ref{cor:realCKgeneral}. Again, such inequalities may be useful in the study of weak turbulence~\cite{bredendesvillettes}.

\begin{proposition}  \label{theorem:generalCK}  
Let some $(\A$,~$\mathrm{Bor}(\Bar{J}))$-measurable $\varpi : \E \to \Bar{J}$ and $r \in [1,2]$. Then for any $(\A$,~$\mathrm{Bor}(\Bar{J}))$-measurable $f,g : \E \to \Bar{J}$ such that all terms below are finite, we have 
\begin{equation} \label{eq:CKgeneral}
    \left\| (f-g) \, \varpi \right\|_{L^r(g \in J)}^2 \leq \left(\int_0^1 (1-\tau) \, \left\|\frac{ \varpi^2}{\Phi''((1-\tau)g + \tau f) } \right\|_{L^{\frac{r}{2-r}}(g \in J)}^{-1} \, \dd \tau \right)^{-1} \, \HH_{\Phi}[f|g],
\end{equation}
where $\HH_{\Phi}$ is the $\Phi$-relative-entropy defined in \eqref{eq:phirelatentrop}.
\end{proposition}

\begin{proof} Let us recall the definition of $\HH_{\Phi}[f|g]$, that is
$$
\HH_{\Phi}[f|g] = \int_0^1 (1-\tau) \left( \int_{g \in J} (f-g)^2 \, \Phi''((1-\tau)g + \tau f) \, \dd \mu(\zeta) \right) \dd \tau.
$$
We fix $\tau \in (0,1)$. Let $p = \frac{2}{r}\in[1,2]$ and $q = \frac{2}{2-r} \in [2,+\infty]$, so that $\displaystyle \frac{1}{p} + \frac{1}{q} = 1$. By Hölder's inequality, we have
\begin{align*}
\int_{g \in J} |f-g|^{\frac{2}{p}} \, \varpi^{\frac{2}{p}} \, \dd \mu(\zeta) &= \int_{g \in J} |f-g|^{\frac{2}{p}}\Phi''((1-\tau)g + \tau f)^{\frac{1}{p}} \, \left[ \frac{\varpi^2}{\Phi''((1-\tau)g + \tau f)} \right]^{\frac{1}{p}} \, \dd \mu(\zeta) \\
&\leq \left(\int_{g \in J} (f-g)^2 \, \Phi''((1-\tau)g + \tau f) \, \dd \mu(\zeta) \right)^{\frac{1}{p}} \, \left\|\left[\frac{\varpi^2}{\Phi''((1-\tau)g + \tau f)} \right]^{\frac{1}{p}} \right\|_{L^q(g \in J)}.
\end{align*}
Raising the above inequality to the power $p$, we obtain
\begin{equation} \label{eq:proofCK1}
\left(\int_{g \in J} |f-g|^{\frac{2}{p}}  \, \varpi^{\frac{2}{p}} \, \dd \mu(\zeta) \right)^p  \leq \left(\int_{g \in J} (f-g)^2 \, \Phi''((1-\tau)g + \tau f) \, \dd \mu(\zeta) \right) \, \left\|\frac{\varpi^2}{\Phi''((1-\tau)g + \tau f)} \right\|_{L^{\frac{q}{p}}(g \in J)},
\end{equation}
where we used the fact that $\|\cdot^{\frac{1}{p}} \|^p_{L^q} = \|\cdot\|_{L^{\frac{q}{p}}}$. Since $\frac{2}{p} = r$ and $\frac{q}{p} = \frac{r}{2-r}$, \eqref{eq:proofCK1} actually writes
\begin{equation*}
\left\| (f-g) \, \varpi \right\|_{L^r(g \in J)}^2  \leq \left(\int_{g \in J} (f-g)^2 \, \Phi''((1-\tau)g + \tau f) \, \dd \mu(\zeta) \right) \, \left\|\frac{\varpi^2}{\Phi''((1-\tau)g + \tau f)} \right\|_{L^{\frac{r}{2-r}}(g \in J)}.
\end{equation*}
If $\varpi$ is zero $\mu$-almost everywhere on $\{g \in J\}$, then the proposition is trivial. Else, since $\Phi'' > 0$ on $\{g \in J \}$ and $(1-\tau)g + \tau f \in J$ on $\{g \in J\}$ for any $\tau \in (0,1)$, we know that $\left\|\Phi''((1-\tau)g + \tau f) \, \varpi^2 \right\|_{L^{\frac{r}{2-r}}(g \in J)} > 0$. Since we assumed that the quantity
$$
\left(\int_0^1 (1-\tau) \, \left\|\frac{ \varpi^2}{\Phi''((1-\tau)g + \tau f) } \right\|_{L^{\frac{r}{2-r}}(g \in J)}^{-1} \, \dd \tau \right)^{-1}
$$
is finite, we also know that for almost every $\tau \in (0,1)$ we have $\displaystyle \left\|\frac{\varpi^2}{\Phi''((1-\tau)g + \tau f)} \right\|_{L^{\frac{r}{2-r}}(g \in J)} < \infty$. For these values of $\tau$, we then have
\begin{equation*}
\left\| (f-g) \, \varpi \right\|_{L^r(g \in J)}^2 \; \frac{1-\tau}{\left\|\frac{\varpi^2}{\Phi''((1-\tau)g + \tau f)} \right\|_{L^{\frac{r}{2-r}}(g \in J)}} \leq (1 - \tau) \left(\int_{g \in J} (f-g)^2 \, \Phi''((1-\tau)g + \tau f) \, \dd \mu(\zeta) \right).
\end{equation*}
Integrating in $\tau$ yields
$$
\left\| (f-g) \, \varpi \right\|_{L^r(g \in J)}^2 \left( \int_0^1 (1-\tau) \, \left\|\frac{ \varpi^2}{\Phi''((1-\tau)g + \tau f) } \right\|_{L^{\frac{r}{2-r}}(g \in J)}^{-1} \, \dd \tau \right) \leq \HH_{\Phi}[f|g].
$$
Since, again by hypothesis, the integral in $\tau$ is nonzero (its inverse is finite), we obtain \eqref{eq:CKgeneral}.
\end{proof}

\noindent Remarking that $\{M_{\Phi}^{\F} \in J\} = \E$ and that for any $f\in \F$, $\HH_{\Phi}[f|M_{\Phi}^{\F}] = H_{\Phi}[f|M_{\Phi}^{\F}]$ (see Proposition~\ref{prop:general}), we straightforwardly obtain the following corollary.

\begin{corollary} \label{cor:realCKgeneral}
Let some $(\A$,~$\mathrm{Bor}(\Bar{J}))$-measurable $\varpi : \E \to \Bar{J}$ and $r \in [1,2]$. With the same notations as in Proposition~\ref{prop:general}, assuming $M_{\Phi}^{\F}$ exists, we have for any $f \in \F$ such that the integral term below is finite, 
\begin{equation} \label{eq:realCKgeneral}
    \left\| (f-M_{\Phi}^{\F}) \, \varpi \right\|_{L^r}^2 \leq \left(\int_0^1 (1-\tau) \, \left\|\frac{ \varpi^2}{\Phi''((1-\tau)M_{\Phi}^{\F} + \tau f) } \right\|_{L^{\frac{r}{2-r}}}^{-1} \, \dd \tau \right)^{-1} \, H_{\Phi}[f|M_{\Phi}^{\F}],
\end{equation}
where $H_{\Phi}$ is defined in \eqref{eqdef:philantrop}--\eqref{eqdef:relatphilantrop} and $M_{\Phi}^{\F}$ is the equilibrium associated to $\Phi$ and the set $\F$, defined in \eqref{eq:generalequilibrium}.
\end{corollary}

\section{Technical results} \label{appendix:technicallemmas}

\noindent In this section, we consider $\Me$, the Fermi-Dirac distribution associated to some $\e>0 $ and $0 \leq f \in L^1_2(\R^3)$ such that $1 - \varepsilon f \geq 0$. The existence of $\Me$ is provided by assuming $\displaystyle \gamma > \frac25$ (see~\cite{lu2001spatially}), where we recall the notation
\begin{equation} \label{eqdefnotation:gamma}
\gamma := \frac{T}{T_F(\rho, \e)},
\end{equation}
where $\displaystyle T_F(\rho, \e) = \frac12 \left(\frac{3 \rho \varepsilon}{4 \pi} \right)^{2/3}$ is the Fermi temperature associated to $\rho$ and $\varepsilon$; and $\rho, T$ are respectively the density and temperature associated to the distribution $f$, defined in \eqref{eq:normalizationdef}. We also recall the notation, for $x \in [0,\e^{-1})$,
$$
\displaystyle \varphi_{\e}(x) = \frac{x}{1- \e x}.
$$
\subsection{\texorpdfstring{$L^\infty$ bound for the Fermi-Dirac statistics}{Linf bound for the Fermi-Dirac statistics}} In this subsection, we provide an $L^\infty$ bound on the Fermi-Dirac statistics. The following result is very similar to \cite[Lemma~A.1]{ABL}.
\begin{proposition} \label{prop:kappa1}
Let $\e>0$ and $f \in L^1_2(\R^3)$ be a nonnegative distribution such that $1 - \varepsilon f \geq 0$ and $\displaystyle \gamma \geq \gamma^{\dag}$, where $\gamma$ is given by~\eqref{eqdefnotation:gamma} and 
\begin{equation} \label{eqdefprop:gammadag}
   \gamma^{\dag} := \left(\frac{4}{\pi}\right)^{\frac13} \left( \frac{5}{3} \right)^{\frac{5}{3}}.
\end{equation}
Then the quantity $\varepsilon \| \varphi_{\varepsilon}(\Me) \|_{\infty}$ satisfies
\begin{equation} \label{eq:propkappa11}
\varepsilon \| \varphi_{\varepsilon}(\Me) \|_{\infty} \leq \frac23 \left(\frac{\gamma}{\gamma^{\dag}} \right)^{-3/2}. 
\end{equation}
\end{proposition}

\begin{proof}
Our proof is based on~\cite[proof of Proposition 3]{lu2001spatially}. We introduce, for $s \geq 0$,
$$
I_s(t) := \int_0^{\infty} \frac{r^s}{1 + t e^{r^2}} \, \dd r, \qquad P(t) := I_4(t) [I_2(t)]^{-5/3}, \quad t>0.
$$
It is proven in~\cite[proof of Proposition 3]{lu2001spatially} that $P$ is continuous and increasing on $\R_+$, and that
$$
P \left( \frac{1}{\varepsilon  \| \varphi_{\varepsilon}(\Me) \|_{\infty}} \right) = 3 \rho^{-2/3} \, T \, (4 \pi \varepsilon^{-1})^{2/3} \equiv \frac{3^{5/3}}{2} \; \gamma.
$$
Let $\alpha > 0$ and
$$
t = \alpha^{-1} \times \frac{3 \sqrt{\pi}}{4} \, \gamma^{3/2}.
$$
As, for any $r \geq 0$, it holds that
$$
\frac{e^{-r^2}}{1+t} \leq \frac{1}{1+t \, e^{r^2}} \leq \frac{e^{-r^2}}{t},
$$
we have
$$
P(t) \leq \left(\frac{1}{t} \int_0^{\infty} r^4 \, e^{-r^2} \, \dd r \right) \left(\frac{1}{1+t} \int_0^{\infty} r^2 \, e^{-r^2} \, \dd r \right)^{-5/3} = \frac{(1+t)^{5/3}}{t} \, \left( \frac12 \Gamma \left(\frac{5}{2} \right) \right)  \left( \frac12 \Gamma \left(\frac{3}{2} \right) \right)^{-5/3},
$$
that is
$$
P(t) \leq  \frac{(1+t)^{5/3}}{t} \times \frac{3 \times 2^{1/3}}{ \pi^{1/3}}.
$$
We define
$$
\gamma^{\alpha} := \left(\frac{4}{3 \sqrt{\pi}} \times \frac{\alpha}{\alpha^{2/5}-1} \right)^{2/3}.
$$
Then, whenever $\gamma \geq \gamma^{\alpha}$, we have $\displaystyle t \geq \frac{1}{\alpha^{2/5} - 1}$, so that $1 + t \leq \alpha^{2/5} \,  t$, implying
\begin{align*}
P(t) \leq t^{2/3} \times \frac{3 \times 2^{1/3} \, \alpha^{2/3}}{\pi^{1/3}} =  \alpha^{-2/3} \times \left(\frac{3 \sqrt{\pi}}{4} \right)^{2/3} \gamma \times \frac{3 \times 2^{1/3} \, \alpha^{2/3}}{\pi^{1/3}} =  \frac{3^{5/3}}{2} \, \gamma = P \left( \frac{1}{\varepsilon  \| \varphi_{\varepsilon}(\Me) \|_{\infty}} \right).
\end{align*}
Since $P$ is increasing, we deduce that, whenever $\gamma \geq \gamma^{\alpha}$, we have
$$
\alpha^{-1} \times \frac{3 \sqrt{\pi}}{4} \, \gamma^{3/2} = t \leq \frac{1}{\varepsilon  \| \varphi_{\varepsilon}(\Me) \|_{\infty}},
$$
that is
\begin{equation} \label{eq:andthisprovesC}
\varepsilon \| \varphi_{\varepsilon}(\Me) \|_{\infty} \leq \alpha \times \frac{4}{3 \sqrt{\pi}} \gamma^{-3/2}.
\end{equation}
By computing the derivative of $\displaystyle \alpha \mapsto \frac{\alpha}{\alpha^{2/5}-1}$, we can minimize $\alpha \mapsto \gamma^{\alpha}$ and find that the minimum value is
$$
\gamma^{\dag} = \left(\frac{4}{\pi}\right)^{\frac13} \left( \frac{5}{3} \right)^{\frac{5}{3}},
$$
reached for $\alpha^{\dag} = \left( \frac{5}{3} \right)^{5/2}$, and, combined with~\eqref{eq:andthisprovesC}, this proves \eqref{eq:propkappa11}.

\medskip

\end{proof}

\bigskip

\subsection{\texorpdfstring{Regularity in $\e$ of the coefficients of the Fermi-Dirac statistics}{Regularity in epsilon of the coefficients of the Fermi-Dirac statistics}} In this last subsection, for any $\varepsilon \geq 0$ and $0 \leq g \in L^1_2(\R^3) \cap L \log L (\R^3)$, we denote by $\M_{\e} \equiv \M_{\e}^{\vi_{\e}^{-1}(g)}$ the $\e$-Fermi distribution associated to $\vi_{\e}^{-1}(g)$. In particular, in the limit case $\e = 0$, $\M_0$ is the Maxwellian distribution associated to $g$. We also denote $a_{\e}, b_{\e}, \bar{u}_{\e}$ and $\rho_{\e}, u_{\e},T_{\e}$ the quantities such that, letting
\begin{equation} \label{notationsannex}
\M_{\e} \equiv \M_{\e}^{\frac{g}{1 + \e g}}, \qquad M_{\e} = \frac{\M_{\e}}{1 - \e \M_{\e}},
\end{equation}
we have, for any $v \in \R^3$,
\begin{equation} \label{eq:notationMepscoeffs}
M_{\e}(v) = \exp \left(a_{\e} + b_{\e} |v-\Bar{u}_{\e}|^2 \right),
\end{equation}
and
\begin{equation} \label{notationsannex2}
\int_{\R^3} \frac{g}{1 + \e g} \begin{pmatrix}
       1 \\ v \\ |v|^2
\end{pmatrix} \dd v = \begin{pmatrix}
       \rho_{\e} \\ \rho_{\e} u_{\e} \\ 3 \rho_{\e} T_{\e} + \rho_{\e} |u_{\e}|^2
   \end{pmatrix}.
\end{equation}

\begin{lemma} \label{lemma:continuityeasy}
Let $0 \leq g \in L^1_2(\R^3) \cap L \log L(\R^3)$. Using the notation~\eqref{notationsannex2}, the application $\e \mapsto (\rho_{\e},u_{\e},T_{\e})$ is continuous on $\R_+$ and $\mathcal{C}^1$ on $\R_+^*$.
\end{lemma}
\begin{proof}
The continuity of $\e \mapsto (\rho_{\e},u_{\e},T_{\e})$ on $\R_+$ comes by dominated convergence, as it holds for any $\e \geq 0$ and $v \in \R^3$ that
$$
\frac{g}{1 + \e g} \, (1 + |v|^2)  \leq  g \, (1 + |v|^2),
$$
and by hypothesis $0 \leq g \in L^1_2(\R^3)$. Similarly, for any $\e > 0$ and $v \in \R^3$ we have
$$
\left|\partial_{\e} \frac{g}{1 + \e g}\right| \, (1 + |v|^2) = \left(\frac{g}{1 + \e g}\right)^2 \, (1 + |v|^2) \leq \e^{-1} \, g \, (1 + |v|^2).
$$
Therefore, for any $\e > 0$, we have
$$
\sup_{\e_* \in (\e/2,3\e/2)}\left|\partial_{\e_*} \frac{g}{1 + \e_* g}\right| \, (1 + |v|^2) \leq 2 \e^{-1} \, g \, (1 + |v|^2).
$$
The differentiability of $\e \mapsto (\rho_{\e},u_{\e},T_{\e})$ on $\R_+^*$ then comes by dominated convergence, as $g \in L^1_2(\R^3)$.
\end{proof}

\noindent The following lemmas provide the continuity of $\e \mapsto (a_{\e}, b_{\e}, \bar{u}_{\e})$ at the point $\e = 0$ and its differentiability on $\R_+^*$.

\begin{lemma} \label{lemma:continuity}
Let $0 \leq g \in L^1_2(\R^3) \cap L \log L(\R^3)$. Using the notations~\eqref{notationsannex}--\eqref{eq:notationMepscoeffs}, the application $\e \mapsto (a_{\e}, b_{\e}, \bar{u}_{\e})$ is continuous at the point $\e = 0$.
\end{lemma}
\begin{proof}
Using the notations~\eqref{notationsannex} and~\eqref{notationsannex2}, we have
\begin{align}
\left| \int_{\R^3} M_{\e} \begin{pmatrix}
       1 \\ v \\ |v|^2
   \end{pmatrix}  \dd v - \int_{\R^3} \M_{\e} \begin{pmatrix}
       1 \\ v \\ |v|^2
   \end{pmatrix}  \dd v \right| &= \left|\int_{\R^3} \frac{\e M_{\e}^2}{1 + \e M_{\e}} \begin{pmatrix}
       1 \\ v \\ |v|^2
   \end{pmatrix}  \dd v \right| \leq \e \|M_{\e}\|_{\infty} \int_{\R^3} \M_{\e} \begin{pmatrix}
       1 \\ |v| \\ |v|^2
   \end{pmatrix}  \dd v \nonumber \\
   &\leq \e \|M_{\e}\|_{\infty} (\rho_{\e} +  3 \rho_{\e} T_{\e} + \rho_{\e} |u_{\e}|^2). \label{eq:prooflemmacontinuity}
\end{align}
Recall the notation, in this case,
$$
\gamma_{\e} = \frac{T_{\e}}{T_F(\rho_{\e}, \e)},
$$
where $\displaystyle T_F(\rho_{\e}, \e) = \frac12 \left(\frac{3 \rho_{\e} \, \varepsilon}{4 \pi} \right)^{2/3}$. By continuity at the point $\e = 0$ of the application $\e \mapsto (\rho_{\e},T_{\e})$, given by Lemma~\ref{lemma:continuityeasy}, we have $$\gamma_{\e} \xrightarrow{\e \to 0} + \infty.$$ 

\smallskip

\noindent Thereby, there exists $\e^* > 0$ such that $\gamma_{\e} \geq \gamma^{\dag}$ for any $\e \in (0,\e^*)$, where $\gamma^{\dag}$ is a universal constant defined in Proposition~\ref{prop:kappa1}. Then, from~\eqref{eq:propkappa11} in Proposition~\ref{prop:kappa1}, for any $\e \in (0,\e^*)$,
$$
\e \|M_{\e}\|_{\infty} \leq \frac23 \left( \frac{\gamma_{\e}}{\gamma_*}\right)^{-\frac{3}{2}},
$$
which vanishes as $\e \to 0$, since $\gamma_{\e}$ tends to $+ \infty$ in this limit. Combining this result, the continuity of $\e \mapsto (\rho_{\e},u_{\e},T_{\e})$ at $\e = 0$, given by Lemma~\ref{lemma:continuityeasy}, and Equation~\eqref{eq:prooflemmacontinuity}, we obtain
$$
\left| \int_{\R^3} M_{\e} \begin{pmatrix}
       1 \\ v \\ |v|^2
   \end{pmatrix}  \dd v - \int_{\R^3} \M_{\e} \begin{pmatrix}
       1 \\ v \\ |v|^2
   \end{pmatrix}  \dd v \right| \xrightarrow{\e \to 0} 0.
$$
The continuity of $\e \mapsto (\rho_{\e},u_{\e},T_{\e})$ at $\e = 0$ being equivalent to the statement
$$
 \int_{\R^3} \M_{\e} \begin{pmatrix}
       1 \\ v \\ |v|^2
   \end{pmatrix}  \dd v \xrightarrow{\e \to 0} \int_{\R^3} \M_0 \begin{pmatrix}
       1 \\ v \\ |v|^2
   \end{pmatrix}  \dd v,
$$
we finally conclude that
$$
\int_{\R^3} M_{\e} \begin{pmatrix}
       1 \\ v \\ |v|^2
   \end{pmatrix}  \dd v \xrightarrow{\e \to 0} \int_{\R^3} \M_{0} \begin{pmatrix}
       1 \\ v \\ |v|^2
   \end{pmatrix}  \dd v.
$$
Both $M_{\e}$ and $\M_0$ are gaussian distributions, which coefficients are continuously defined by the above moments, allowing to conclude to the continuity of these coefficients at the point $\e=0$.
\end{proof}

\begin{lemma} \label{lemma:derivMeps}
Let $0 \leq g \in L^1_2(\R^3) \cap L \log L(\R^3)$. Using the notations~\eqref{notationsannex}--\eqref{eq:notationMepscoeffs}, the application $\e \mapsto (a_{\e}, b_{\e}, \bar{u}_{\e})$ is $\mathcal{C}^1$ on $\R_+^*$. 
\end{lemma}

\begin{proof}
Since the distribution $M_{\e}(\cdot + \Bar{u}_{\e})$ is radially symmetric, so is $\M_{\e}(\cdot + \Bar{u}_{\e})$, hence $\Bar{u}_{\e} = u_{\e}$, which, by Lemma~\ref{lemma:continuityeasy}, is $\mathcal{C}^1$ on $\R_+^*$. We then define
$$
g_{\e} : = g(\cdot + u_{\e}), \qquad \mathcal{N}_{\e} := \M_{\e}(\cdot + u_{\e}) \equiv \M_{\e}^{\vi_{\e}^{-1}(g_{\e})} \quad \text{and} \quad N_{\e} := M_{\e}(\cdot + u_{\e}),
$$
so that, for any $\e > 0$, it holds that
$$
N_{\e} = \exp \left(a_{\e} + b_{\e} |v|^2 \right).
$$
As in~\eqref{notationsannex2}, we let $\rho_{\e}, T_{\e} > 0$ be such that
$$
\int_{\R^3} \frac{g_{\e}}{1 + \e g_{\e}} \begin{pmatrix}
       1 \\ v \\ |v|^2
\end{pmatrix} \dd v = \begin{pmatrix}
       \rho_{\e} \\ 0 \\ 3 \rho_{\e} T_{\e}
   \end{pmatrix}.
$$
Let us now show that $\e \mapsto (a_{\e}, b_{\e})$ is $\mathcal{C}^1$ on $\R_+^*$. It is proven in~\cite[proof of Proposition 3]{lu2001spatially} that, letting
$$
I_s(\tau) := \int_{0}^{\infty} \frac{r^s}{1 + \tau \, e^{r^2}} \, \dd r, \qquad P(\tau) := I_4(\tau) \, I_2(\tau)^{-5/3}, \qquad \tau \in \R_+^*,
$$
the function $P$ is an increasing $\mathcal{C}^1$ function from $\R_+^*$ to $\left(\frac{3^{5/3}}{5}, +\infty \right)$, with $P' > 0$ on $\R_+^*$. Therefore it is invertible, and $P^{-1}$ is also $\mathcal{C}^1$. All the more, a dominated convergence argument ensures that $I_2$ is $\mathcal{C}^1$ on $\R_+^*$. It is moreover shown in~\cite[proof of Proposition 3]{lu2001spatially} that
$$
\left(\frac{\e}{4 \pi} \right)^{2/3} P \left( \frac{1}{\e e^{a_{\e}}} \right) = \frac{3 \rho_{\e} T_{\e}}{\rho_{\e}^{5/3}}, \qquad b_{\e} = -\left( \frac{4 \pi }{\e \rho_{\e}} \, I_2\left( \frac{1}{\e e^{a_{\e}}} \right) \right)^{\frac23}.
$$
By Lemma~\ref{lemma:continuityeasy}, the application $\e \mapsto (\rho_{\e}, T_{\e})$ is $\mathcal{C}^1$ on $\R_+^*$, hence so is the application $\e \mapsto (a_{\e},b_{\e})$, as a composition of $\mathcal{C}^1$ applications.
\end{proof}

\begin{lemma} \label{lem:bdMeps}
Let $0 \leq g \in L^1_2(\R^3) \cap L \log L(\R^3)$. Using the notation~\eqref{notationsannex}, for any $\overline{\e} > 0$, there exist $C>0$ and $\eta > 0$ such that for any $\e \in [0,\overline{\e}]$ and $v \in \R^3$, we have
\begin{equation} \label{eqlem:bdMeps1}
M_{\e}(v) \leq C \, e^{- \eta |v|^2},
\end{equation}
and
\begin{equation} \label{eqlem:bdMeps2}
    |\log M_{\e}(v)| \leq C (1 + |v|^2).
\end{equation}
\end{lemma}

\begin{proof}
We denote 
$$
\underline{a} = \sup_{\e \in [0,\overline{\e}]} |a_{\e}|, \qquad \underline{b} = -\sup_{\e \in [0,\overline{\e}]} |b_{\e}|, \qquad \overline{b} = -\inf_{\e \in [0,\overline{\e}]} |b_{\e}| \quad \text{and}  \quad \underline{u} = \sup_{\e \in [0,\overline{\e}]} |u_{\e}|.
$$
Combining the results of Lemmas~\ref{lemma:continuity} and~\ref{lemma:derivMeps}, the application $\e \mapsto (a_{\e},b_{\e},u_{\e})$ is continuous on $\R_+$, from which we deduce that $\underline{a}$, $\underline{b}$ and $\underline{u}$ are finite. 

\smallskip

\noindent Moreover, as $M_{\e} \in L^1(\R^3)$ for all $\e \in [0,\overline{\e}]$, the application $\e \mapsto b_{\e}$ is (strictly) negative on $[0,\overline{\e}]$, so that, as $\e \mapsto b_{\e}$ is continuous on $[0,\overline{\e}]$, we have $\overline{b} < 0$. 

\medskip

\noindent Therefore, for any $0 \leq \e \leq \overline{\e}$ and $v \in \R^3$, we have
$$
|\log M_{\e}(v)| = |a_{\e} + b_{\e}|v-u_{\e}|^2| \leq \underline{a} + 2 |\underline{b}| \,  \underline{u}^2 +  2 |\underline{b}| \, |v|^2,
$$
and, since $|v-u_{\e}|^2 \geq \frac12 |v|^2 - |u_{\e}|^2 \geq \frac12 |v|^2 - \underline{u}^2$ and $\overline{b} < 0$,
$$
M_{\e} (v) = e^{a_{\e} + b_{\e}|v-u_{\e}|^2} \leq e^{\underline{a} + \overline{b}|v-u_{\e}|^2} \leq e^{\underline{a} + \left|\overline{b} \right| \underline{u}^2 + \frac12\overline{b}|v|^2}.
$$
Letting $\eta = -\frac12 \overline{b}$ and $C = \max \left(\underline{a} + 2 |\underline{b}| \,  \underline{u}^2, \,  e^{\underline{a} + \left|\overline{b} \right| \underline{u}^2} \right)$ yields the result.
\end{proof}

\section*{Acknowledgements}
\noindent I thank Bertrand Lods with whom I had fruitful discussions and Matthieu Dolbeault who suggested to me the proof of the nice inequality~\eqref{ineq:Lbda1}.

\bigskip

\noindent I have no competing interests to declare that are relevant to the content of this article.

\bibliographystyle{plain}
\bibliography{biblio}

\end{document}